\documentclass[reqno,11pt]{amsart}
\usepackage{graphicx, amsmath, amssymb, amscd, mathtools, hyperref, amsthm, euscript, caption, amsfonts,bm,bbm,xcolor}  

\hypersetup{
  colorlinks=true,
  linkcolor=blue,
  citecolor=blue,
  urlcolor=blue
}

\usepackage[T1]{fontenc}
\DeclareMathAlphabet{\mathpzc}{OT1}{pzc}{m}{it}

\usepackage{wrapfig}
\usepackage{psfrag, graphicx, pgf, float}
\usepackage{mathrsfs}

\newtheorem{thm}[equation]{Theorem}
\newtheorem{theorem}[equation]{Theorem}

\newtheorem{rmk}[equation]{Remark}

\newtheorem{prop}[equation]{Proposition}

\newtheorem{cor}[equation]{Corollary}
\newtheorem{lem}[equation]{Lemma}\newtheorem{Lem}[equation]{Lemma}
\newtheorem{lemma}[equation]{Lemma}
\newtheorem{Def}[equation]{Definition}

\numberwithin{equation}{section}

\numberwithin{equation}{section}

\newcommand{\be}{begin{equation}}
\newcommand{\bH}{\mathbb H}

\newcommand{\e}{\epsilon}
\newcommand{\z}{\mathbb{Z}}

\renewcommand{\c}{\mathbb{C}}
\newcommand{\br}{\mathbb{R}}

\newcommand{{\grinv}}{{\Cal G}^{-r}}

\newcommand{\ba}{\backslash}

\newcommand{\G}{\Gamma}

\newcommand{\Cal}{\mathcal}

\newcommand{\bp}{\begin{pmatrix}}
\newcommand{\ep}{\end{pmatrix}}
\renewcommand{\be}{\begin{equation}}
\newcommand{\ee}{\end{equation}}

\renewcommand{\bp}{{\rm bp}}

\newcommand{\SO}{\operatorname{SO}}
\newcommand{\SU}{\operatorname{SU}}

\newcommand{\vol}{\operatorname{vol}}

\newcommand{\Vol}{\op{Vol}}
\newcommand{\PSL}{\op{PSL}}

\newcommand{\norm}[1]{\lVert #1 \rVert}

\newcommand{\op}{\operatorname}\newcommand{\supp}{\operatorname{supp}}

\newcommand{\cl}[1]{\overline{#1}}

\newcommand{\Ad}{\operatorname{Ad}}

\newcommand{\Ga}{\Gamma}

\newcommand{\Z}{\z}

\newcommand{\ga}{\gamma}
\newcommand{\al}{\alpha}

\newcommand{\La}{\Lambda}

\def\cM{\cal M}

\def\e{\mathrm{e}}

\def\dim{\operatorname{dim}}

\def\SO{\operatorname{SO}}

\def\PSL{\operatorname{PSL}}

\def\supp{\operatorname{supp}}

\def\vol{\operatorname{vol}}

\def\Stab{\mathrm{Stab}}

\newcommand{\fa}{\mathfrak{a}}

\newcommand{\cal}{\mathcal}
\renewcommand{\e}{\varepsilon}
\renewcommand{\epsilon}{\e}

\newcommand{\wo}{\mathsf w_0}

\newcommand{\core}{\mathrm{core}}
\newcommand{\hull}{\mathrm{hull}}

\newcommand{\RFM}{\mathrm{RF}\cal M}

\newcommand{\sfp}{\mathsf{p}}

\newcommand{\sfd}{\mathsf{d}}
\newcommand{\inj}{\op{inj}}

\newcommand{\xx}{X}
\newcommand{\cm}{\op{core}(\cal M)}

\newcommand{\cN}{\mathcal N}\newcommand{\nc}{\op{nc}}
\begin{document}

\title[Totally geodesic submanifolds]{Properness and finiteness of totally geodesic submanifolds in the convex core}

\author{Minju Lee  and Hee Oh}
\address{Department of Mathematical Sciences, KAIST, Daejeon 34141, Korea}
\email{minju.lee@kaist.ac.kr}
\address{Mathematics department, Yale university, New Haven, CT 06520}

\email{hee.oh@yale.edu}

\begin{abstract} 
    We study totally geodesic submanifolds in the convex core of geometrically finite rank-one locally symmetric manifolds. Although the infinite-volume setting can exhibit highly complicated behavior, including geodesic planes with fractal closures, we show that a strong rigidity persists inside the convex core. This rigidity has  striking consequences in the infinite volume setting: every maximal totally geodesic submanifold of dimension at least two contained in the convex core is properly immersed and has finite volume, and only finitely many such submanifolds can occur. These results stand in sharp contrast to the behavior in the finite-volume setting. Moreover, combining this finiteness result with the work of Bader-Fisher-Miller-Stover and of Gromov-Schoen, we deduce that every geometrically finite rank-one manifold with infinitely many maximal totally geodesic submanifolds of dimension at least two and of finite volume must be arithmetic.
\end{abstract}

\maketitle
\tableofcontents
\section{Introduction}\label{sec.int} The study of totally geodesic submanifolds occupies a central position in differential geometry, dynamics, and geometric topology. In locally symmetric spaces of finite volume, such manifolds arise
naturally from algebraic subgroups and encode deep arithmetic and geometric structure. Their distribution reflects the intricate interplay between discrete subgroups, rigidity phenomena, and homogeneous dynamics. 

In the infinite volume setting, however, the picture is
considerably more delicate. Although geometric finiteness provides a natural analogue of finite volume, many of the structural features familiar from
the compact or finite volume world no longer hold. Even in real hyperbolic $3$-manifolds, the qualitative behavior of geodesic planes depends sensitively on the  geometry of the ambient space, and examples are known in which
certain geodesic planes have fractal closures.

The main purpose of this paper is to establish that, despite these potential pathologies, a robust form of rigidity persists for totally geodesic submanifolds contained in the convex core of a geometrically finite rank-one locally symmetric manifold. This rigidity has striking implications for the geometry 
of infinite volume geometrically finite manifolds, in sharp contrast to what can occur in the finite volume case.

\medskip 

Let  $G$ be a connected simple real algebraic group of rank one, and let $X=X_G$ denote its associated Riemannian symmetric space.  Concretely, $G$ is locally isomorphic to one of $\SO(d,1)$, $\SU(d, 1)$, $\op{Sp}(d, 1)$  $(d\geq 2)$, or $F_4^{-20}$
and $X$ is the real $\mathbb H^d_{\mathbb R}$, complex $\mathbb H^d_{\mathbb C}$, quaternionic $\mathbb H^d_{\mathbb H}$, or octonionic hyperbolic space $\mathbb H^2_{\mathbb O}$, respectively.

Let $\Gamma<G$ be a non-elementary torsion-free discrete subgroup, and let $\cal M=\Ga\ba X$ be the associated locally symmetric space.
The convex core of $\cal M$ is the smallest closed convex subset of $\cal M$ 
whose inclusion into $\cM$ is a homotopy equivalence.
Equivalently, 
$$\core (\cM)=\Ga\ba \hull (\La)$$
where $\La$ denotes the limit set of $\Ga$ and
$\hull (\La)\subset X$ its convex hull (see Section \ref{sec.gf}).

See Figure~\ref{fig:gf} for an illustration.

\begin{figure}[htbp] \begin{center}
  \includegraphics [height=4cm]{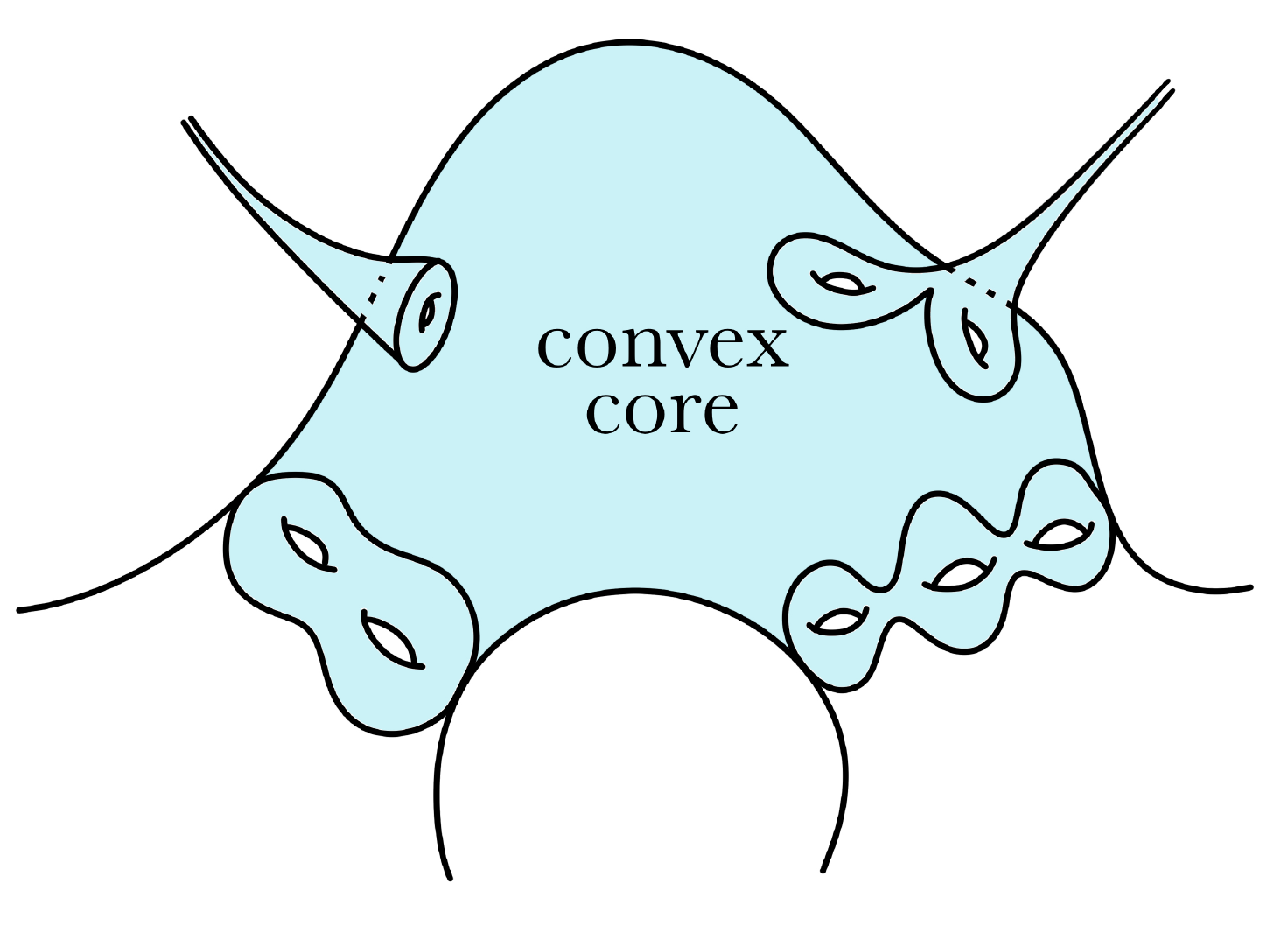} 
\end{center}
\caption{}
\label{fig:gf}
\end{figure}    

Throughout this paper, we assume that
$$\text{$\cal M$ is geometrically finite},$$
that is, the unit neighborhood of $\op{core}(\cal M)$ has finite volume. Clearly, every finite-volume rank-one manifold is geometrically finite, but the class of geometrically finite manifolds is much broader than that of finite-volume ones.

\subsection*{Properness of totally geodesic submanifolds} 
A totally geodesic subspace of $X$ is a connected smooth submanifold $Y$ that contains every complete geodesic in $X$ tangent to $Y$. A totally geodesic submanifold $\cal N$ of $\cal M$ is the image of such a $Y$ under the quotient map $X\mapsto \Ga\ba X$.
Equivalently, 
$\cal N$ is the image of the natural immersion $$\iota: \op{Stab}_\Ga (Y) \ba Y \to \cal M$$  where $\op{Stab}_\Ga (Y)= \{\ga\in \Ga: \ga Y =Y\}$. The dimension of $\cal N$ is defined to be the dimension of $Y$, and the volume of $\cal N$ is defined as the volume of  the quotient $\op{Stab}_\Ga (Y) \ba Y$. We say that $\mathcal N$ is properly immersed if the immersion $\iota$ is  proper; equivalently, if $\cN$ is closed in $\cM$, that is, $\overline{\cN}=\cN$ (Lemma \ref{OS}).

Throughout the paper, a totally geodesic submanifold of $\cM$ is assumed to have 
$${\text{\bf dimension at least two}}$$
so that we need not repeat this assumption.
A totally geodesic submanifold of dimension one will simply be referred to as a geodesic.

We focus on totally geodesic submanifolds contained in $\core (\cM)$.
For instance, every bounded totally geodesic submanifold is automatically contained in $\core (\cM)$.
 By a maximal totally geodesic submanifold of $\core (\cM)$, we mean a totally geodesic submanifold $\cN$ contained in  $\core (\cM)$ with $\cN\ne \cM$ that is maximal with respect to inclusion: if $\cal N'$ is a totally geodesic submanifold
satisfies $\cal N\subset \cal N'\subset \core (\cM)$ and $\cal N'\ne\cM$, then $\cal N=\cal N'$.
\begin{theorem}[Properness]\label{p} Let
  $\cal M$ be a geometrically finite manifold of infinite volume.
Then every maximal totally geodesic immersed submanifold contained in $\core (\cM )$ is properly immersed and has finite volume.
\end{theorem}
If $\cM$ has finite volume, then
$\cal M=\core (\cM)$, and the geodesic flow  is ergodic with respect to the Liouville measure on the unit tangent bundle. In this case,
 almost every totally geodesic submanifold is dense in $\cM$, so Theorem \ref{p} fails completely in the finite volume setting. 
 
 Theorem \ref{p} has several immediate consequences.
First, it implies that if there exists a maximal totally geodesic submanifold $\cal N$ contained in $\core (\cM)$ with $\overline{\cal N}\ne \cal N$,
 then $\cM$ must have finite volume. Second, it follows that $\cM$ contains a finite volume totally geodesic  submanifold if and only 
if there exists a totally geodesic submanifold of dimension two contained in $\core (\cM)$.

The minimal codimension of totally geodesic subspaces in $X=\mathbb H^d_{\mathbb K}$ is equal to $\op{dim}_\br(\mathbb K)$, which is  $1,2,4, 8$ for $\mathbb K=\br, \mathbb C, \mathbb H, \mathbb O$, respectively.
\begin{cor}\label{cor.minimalcodimension} 
Every totally geodesic submanifold of $\cM=\Ga\ba \mathbb H^d_{\mathbb K}$ of codimension $\op{dim}_\br(\mathbb K)$ contained in $\core (\cM)$ is properly immersed and has finite volume.
\end{cor}

\subsection*{Finiteness of totally geodesic manifolds}
We next establish a finiteness property for totally geodesic submanifolds contained in the convex core.

\begin{theorem}[Finiteness I]\label{fin1}
 If $\cal M$ is a geometrically finite manifold of infinite volume, then there exist only finitely many maximal totally geodesic submanifolds  contained in $\core (\cM)$. 
\end{theorem}
By Theorem \ref{p}, all of these submanifolds are properly immersed and have finite volume.

Although there may exist infinitely many  maximal properly immersed totally geodesic submanifolds in $\cM$, only finitely many of them have finite volume:
\begin{theorem}[Finiteness II]\label{fin2}
 If $\cal M$ is a geometrically finite manifold of infinite volume, then there exist only finitely many  maximal totally geodesic submanifolds of finite volume.
\end{theorem}

Theorems \ref{fin1} and \ref{fin2} fail dramatically in the finite volume setting, since
an arithmetic manifold containing one maximal totally geodesic submanifold of finite volume necessarily contains infinitely many such submanifolds. For example, the figure-eight knot complement contains infinitely many finite area immersed totally geodesic surfaces
(see \cite{Re}, \cite[Section~9.5]{MR}).

\begin{rmk}\label{rmk.previousfiniteness} \rm For real hyperbolic $3$-manifolds, these theorems were proved by McMullen--Mohammadi--Oh \cite{MMO1} in the convex cocompact case and by Benoist--Oh \cite{BO1} in full generality. In higher dimensions, they were established in \cite{LO} for convex cocompact real hyperbolic $d$-manifolds whose convex cores have totally geodesic boundary, for all $d \ge 3$.
\end{rmk}

\subsection*{Arithmeticity from totally geodesic submanifolds} 
Every quaternionic or octonionic hyperbolic manifold $\cM=\Ga\ba X$ of finite volume is arithmetic, that is, $\Ga$ is an arithmetic lattice of $G$, by the theorem of Gromov-Schoen \cite{GS}.
 Answering a question of McMullen and Reid, Bader-Fisher-Miller-Stover (\cite{BF1}, \cite{BF2}) showed that if a real or complex hyperbolic manifold of finite volume 
contains infinitely many maximal totally geodesic submanifolds of finite volume, then it is arithmetic. See also  Margulis-Mohammadi \cite{MM} for compact real hyperbolic $3$-manifolds, Baldi-Ullmo \cite{BU} for related results in complex hyperbolic manifolds, and Filip-Fisher-Lowe \cite{FFL} for analogous results in the setting of closed real-analytic manifolds of negative sectional curvature. 

Combining these results with Theorem \ref{fin2} gives the following:
\begin{theorem}\label{thm.arithmeticity}
   If a geometrically finite manifold $\cM$ contains infinitely many\footnote{Throughout the paper, ``infinitely many'' means infinitely many pairwise distinct.} maximal totally geodesic submanifolds of finite volume, then $\cM$ is arithmetic.
\end{theorem}

\subsection*{Finiteness of totally geodesic boundaries}
We denote by $\partial_\infty X$ the geometric boundary of $X$, that is,
the set of equivalence classes of geodesic rays, where two rays are identified if they remain within bounded Hausdorff distance of each other.
For a totally geodesic subspace $Y\subset X$ of dimension at least two, we may regard its boundary $\partial_\infty Y$ as a subset of $\partial_\infty X.$
Let $\mathcal S_X$ denote the space of all such boundaries $\partial_\infty Y\subset \partial_\infty X$, equipped with the Chabauty-Hausdorff topology. 

For a totally geodesic submanifold $\cN=\Ga\ba \Ga Y$, we have $\cN\subset \core (\cM)$ if and only if $\partial_\infty Y \subset \La$ where $\La\subset \partial_\infty X$ is the limit set of $\Ga$. Using this criterion, we deduce the following consequence of Theorems \ref{p} and \ref{fin1}:
\begin{theorem}\label{SX} Let $\Ga<G$ be a geometrically finite, non-lattice subgroup of $G$, and $\La\subset \partial_\infty X$ be its limit set. Then
   the collection $$\{S\in \mathcal S_X: S\subset \La \text{ maximal} \}$$ consists of finitely many closed $\Ga$-orbits in $\mathcal S_X$.
\end{theorem}

\subsection*{Counting spheres in the limit set} Specializing to $G=\op{SO}(d+1,1)^\circ$ $(d\ge 2)$, we have
$X=\bH_{\br}^{d+1}$ and $\partial_\infty X= \br^{d}\cup\{\infty\}$.

Theorem \ref{SX} implies that if $\Ga<G$ is a geometrically finite non-lattice subgroup, then the collection of all $(d-1)$-round spheres contained in the limit set $\La$ forms a locally finite\footnote{This means that for every $\e>0$, there are only finitely many spheres in a fixed bounded region of Euclidean radii larger than $\e>0$} sphere packing consisting of finitely many $\Ga$-orbits.
Spheres arising from the totally geodesic boundary components of $\core (\cM)$
may be regarded as {\it visible} spheres, but there may also be {\it invisible} spheres, namely those lying in the limit set but not corresponding to boundary components of $\core (\cM)$. Given this local finiteness, the sphere counting results of (\cite{OS}, \cite[Theorem 7.5]{Oh}, see also \cite{Oh2}) apply and
 counts all such spheres, visible or invisible, contained in the limit set:

\begin{theorem}\label{thm.spherecounting} Let $\Ga<\op{SO}(d+1,1)^\circ$ be a geometrically finite subgroup with $d\ge 2$.  Assume that $\La\subset \br^{d}$ and that $\La$ contains a $(d-1)$-sphere. Then there exists $c>0$, depending only on $\Ga$, such that, as $t\to \infty$, the number $\mathsf N(\La, t)$ of all $(d-1)$-spheres contained in $\La$ of Euclidean radius at least $1/t$ satisfies
$$ \mathsf N(\La, t)\sim c\, {t^{\dim (\La)}}$$
where $\dim (\La) $ is the Hausdorff dimension of $\La$.
\end{theorem}

\begin{rmk}\label{rmk.asymptoticallygeodesic}\rm 
Al Assal and Lowe, building on the works  \cite{MMO1} and \cite{BO1}, proved that a real-hyperbolic geometrically finite $3$-manifold of infinite volume cannot contain infinitely many compact {\it asymptotically geodesic} surfaces \cite{AL}. In view of their methods, we expect an analogous statement in general rank one: every geometrically finite, infinite-volume rank-one manifold cannot contain infinitely many compact asymptotically geodesic maximal submanifolds.
\end{rmk}

\subsection*{On the proof}
We emphasize that as far as we know, there are no direct geometric or topological arguments that establish
the properness and finiteness theorems (Theorems \ref{p} and \ref{fin1}). 
The key point is that these results arise as strict consequences of a much stronger rigidity phenomenon: the only available method for establishing properness and finiteness is to prove a rigidity-equidistribution  theorem for  totally geodesic submanifolds inside the convex core of a geometrically finite manifold (not necessarily of infinite volume). This stronger statement is the foundation for the entire argument:
\begin{theorem} [Theorem \ref{Theoremmax0}, Theorem \ref{eq}] \label{rigid}
    Let $\cM$ be geometrically finite.
  \begin{enumerate}
     \item The closure of every totally geodesic submanifold contained in  $\op{core}(\cal M)$ 
is a totally geodesic submanifold of finite volume.
 \item  If there exist infinitely many maximal totally geodesic submanifolds $\cN_i$ contained in $\core (\cM)$,
    then $\cM$ has finite volume.  If, moreover,  each $\mathcal N_i$ has finite volume,  the sequence ${\cal N_i}$  equidistributes in $\cM$ as $i\to \infty$.
  \end{enumerate}  \end{theorem}

Theorem \ref{rigid} is proved using unipotent dynamics on the homogeneous space $\Ga\ba G$. 
Fix a one-parameter $\mathbb R$-diagonalizable subgroup $A<G$. The relevant subspace of $\Ga \ba G$ for  studying totally geodesic submanifolds contained in $\core (\cM)$ is the smallest closed subset containing all bounded $A$-orbits, which we denote by $\RFM$.
Understanding the closures of totally geodesic submanifolds in $\cal M$ amounts to describing the orbit closures of $xH$ contained in $\RFM$ for all simple non-compact subgroups $H<G$; this is given in Theorem \ref{Theoremmax} where we describe all orbit closures inside $\RFM$ for the action of every connected closed subgroup of $G$ generated by unipotent elements.
By Ratner's arguments \cite{RaGF}, this is reduced to the study of the closure of $xU\subset \RFM$ for
 a one-parameter unipotent subgroup $U=\{u_s\}$ of $H$ (see the proof of Theorem \ref{Theoremmax}).
 We use
the geometric finiteness hypothesis together with the non-divergence of unipotent flows  in the infinite volume rank one setting due to Benoist-Oh \cite{BO1} and Buenger-Zheng  \cite{BZ} to deduce that every weak-$^*$ limit $\nu$ of the sequence $$\nu_T=\frac{1}{T}\int_0^T \delta_{xu_s} \, ds ,$$ is a probability measure on $\Ga \ba G$, where $\delta_{xu_s}$
denotes the Dirac measure at ${xu_s}\in \Ga \ba G$.
We then apply  Ratner's measure classification theorem \cite{Ra} and the avoidance theorem of Dani-Margulis \cite{DM} to the ergodic components of $\nu$ to show that the closure of $xU$ is a homogeneous space. This constitutes the main step of the proof of Theorem \ref{rigid}(1).

For the proof of Theorem \ref{rigid}(2), we first deduce from Theorem \ref{rigid}(1) that every maximal totally geodesic submanifold inside $\core (\cM)$ gives rise to an $H$-invariant {\it probability} measure in $\Ga\ba G$ supported on a closed orbit $x\op{N}_G(H)$, where $\op{N}_G(H)$ denotes the normalizer of $H$, and
apply the Mozes-Shah theorem \cite{MS} on the limiting behavior of probability measures invariant under unipotent flows.

The earlier papers \cite{MMO1} and \cite{BO1} on geometrically finite real hyperbolic $3$-manifolds relied on topological arguments for unipotent dynamics.  This was feasible because, in that specific setting, all totally geodesic submanifolds have codimension one in the ambient manifold, a feature unique to real hyperbolic $3$-manifolds. In more general settings, where totally geodesic submanifolds have varying dimensions or maximal geodesic planes have codimension greater than one, extending topological methods becomes quite complicated. In contrast, our current paper leverages the full power of the measure rigidity theorems of Ratner (\cite{Ra}, \cite{Ra1}, \cite{RaGF}), Dani-Margulis \cite{DM}, and Mozes-Shah \cite{MS} on unipotent flows, allowing us to obtain, in our view,  rather surprising results in the geometry of geometrically finite manifolds.

Finally, we note that the geometric finiteness assumption on $\cM$ is essential: if $\cM$ is a $\Z$-cover of a compact hyperbolic $3$-manifold, then $\cM=\core (\cM)$, and
$\cM$ contains infinitely many dense geodesic surfaces.

\subsection*{Organization}
\begin{itemize}
    \item In Section~\ref{sec.b}, we review the structure of closed subgroups of rank-one Lie groups and introduce the collections $\mathscr H$ and $\mathscr H^*$ of closed subgroups of $G$ stabilizing  totally geodesic subspaces of $X$. 
    \item In Section~\ref{sec.submanifolds}, we relate these to immersed totally geodesic submanifolds of $\cM$. 
    \item In Section~\ref{sec.gf}, we characterize those contained in the convex core of $\cM$ via the renormalized frame bundle $\RFM$.
 \item Sections~\ref{sec.orbitclosures}--\ref{sec.equidistribution} develop the orbit closure and equidistribution results for subgroups generated by unipotent one-parameter subgroups within $\RFM$.
 \item  Section~\ref{sec.rigidity} applies these results to prove the rigidity, properness and finiteness theorems announced in the introduction. 

\end{itemize}

\subsection*{Acknowledgements}
We would like to thank Subhadip Dey, Curtis McMullen and Yair Minsky for helpful comments.
Oh is partially supported by the NSF grant No. DMS-2450703.

\section{Closed subgroups of rank one Lie groups}\label{sec.b}
Let $G$ be a connected simple real algebraic group of rank one. 
In this section, we describe the structure of closed subgroups of $G$.
In particular, we introduce two natural families of subgroups $\mathscr H$ and $\mathscr H^*$ which will play a central role throughout the paper.

Let $\frak g$ denote the Lie algebra of $G$. Let $\theta$ be a Cartan involution of $\frak g$, and $\Theta$ the corresponding involution of $G$ which induces $\theta$.
Write $\frak g$ as the direct sum of $\pm 1$ eigenspaces of $\theta$:
\be\label{cd} \frak g=\frak k\oplus \frak p \ee
where $\frak k=\{\theta(x)=x\}$ and $\frak p=\{\theta(x)=-x\}$.
Let $K< G$ be the maximal compact subgroup with Lie algebra $\frak k$; equivalently, $K$ is the fixed point subgroup of $\Theta$.
Let $B$ denote the Killing form on $\frak g$, and define
the inner product $\langle\cdot, \cdot\rangle$ on $\frak g$:
$$
\langle x, y\rangle :=-B(x, \theta(y)), \quad x,y\in \frak g,
$$
with associated norm $\norm{\cdot}$.
Let $$\text{ $X:=G/K$ and $o=[K]\in X$.}$$
Since $\langle\cdot,\cdot\rangle$ is $\op{Ad}(K)$-invariant, it induces a left $G$-invariant and right $K$-invariant metric on $G$, which descends to a Riemannian metric on $X$, which we denote by $\sfd$.
The Riemannian symmetric space $(X,\sfd) $ is isometric to one of the real, complex, quaternionic, and octonionic hyperbolic spaces.

Let $\fa\subset \frak p$ be a Cartan subalgebra, ie, a maximal abelian subalgebra of $\frak p$.
Since $G$ has real rank one, $\frak a$ is one-dimensional.
Let $A=\exp \fa$ and let $M=\op{C}_K(A)$ denote the centralizer of $A$ in $K$.
We parametrize $A=\{a_t:t\in \br\}$ so that $\sfd(o, a_to)=|t|$ for all $t\in \br$. 
Define the horospherical subgroups
$$ N^\pm=\{g\in G: a_{t}ga_{-t}\to e\text{ as }t\to\pm\infty\}$$
and the corresponding minimal parabolic subgroups
$$
P^\pm=MAN^\pm.
$$
For simplicity, we write $N=N^-$ and $P=P^-$. 
Let $\wo\in\op{N}_K(A)$ represent the non-trivial element of the rank-one Weyl group $\op{N}_K(A)/M$. Thus $\wo P\wo^{-1}=P^+$ and $P\cap P^+=MA=\op{C}_G(A)$.

Denote by $\partial_\infty X$ the boundary of $X$ at infinity; the set of equivalence classes of geodesic rays, where two rays are equivalent if they have finite Hausdorff distance. Every geodesic ray in $X$ is of the form $\{g a_{ t} o: t\ge 0\}$ for some $g\in G$.
The stabilizer of the equivalence class of $\{a_{t} o: t\ge 0\}$ is equal to $P$;
hence we may identify 
$$
G/P \simeq \partial_\infty X,$$
where $gP$ corresponds to the boundary point in $\partial_\infty X$ represented by the ray $\{ga_{ t}o: t\ge 0\}$.

\subsection*{Families $\mathscr{H}^*$ and $\mathscr {H}$}
A totally geodesic subspace $Y$ is a connected smooth submanifold of $X$ which contains all complete geodesics in $X$ tangent to $Y$.
For such $Y\subset X$, we define its stabilizer 
\begin{equation}\label{eq.gy}
G_Y:=\{g\in G: gY=Y\}.    
\end{equation}
We introduce two families of subgroups of $G$ that contain representatives
corresponding to conjugacy classes of $G_Y$. 
 For a subgroup $H<G$,
denote by $H^\circ$ its identity component, and by $\op{C}_G(H)$ and $\op{N}_G(H)$ the centralizer and normalizer of $H$ in $G$, respectively.

We define
$$
\mathscr H^*=\left\{ H < G:
\begin{array}{c}
     \text{ a connected simple $\Theta$-invariant }\\
     \text{non-compact closed subgroup of $G$ containing $A$} 
\end{array}
\right\}
$$
and
\begin{equation}\label{eq.h}
\mathscr H=\{\op{N}_G(H)^\circ: H\in \mathscr H^*\}.
\end{equation}

For a connected reductive subgroup $L<G$, let $L^{\op{nc}}$ denote the maximal connected normal semisimple subgroup of $L$ with no compact factors. If $L^{\op{nc}}\ne\{e\}$, then, since $G$ has real rank one, $L^{\op{nc}}$ is a simple non-compact algebraic subgroup of real rank one. In this case,
$$
L=L^{\op{nc}}\op{C}_L(L^{\op{nc}}),
$$
and $\op{C}_G(L^{\op{nc}})$ (and hence $\op{C}_L(L^{\op{nc}})$) is compact.

The  following lemma  implies that the maps
$$H\mapsto H^{\nc}\quad  \text{ and }\quad  H\mapsto \op{N}_G(H)^\circ$$
define a bijection $\mathscr H\to \mathscr H^*$ and its inverse, respectively.
\begin{lem}\label{lem.h1}
Let $L<G$ be a connected reductive $\Theta$-invariant subgroup containing $A$.
Suppose that $L^{\nc}\ne \{e\}$, and let $H:=\op{N}_G(L^{\op{nc}})^\circ$.  Then  we have
 $$H\in \mathscr H, \; \; H^{\op{nc}}=L^{\op{nc}},\; \;\text{and}\;\;  H.o=L.o.$$
In particular, for every $H\in \mathscr H$, we have
$$H=\op{N}_G(H^{\op{nc}})^\circ\quad\text{and}\quad  H.o=H^{\op{nc}}.o.$$
\end{lem}
\begin{proof} 
Since $L^{\op{nc}}\ne\{e\}$, it is a non-compact normal subgroup of $L$. By the preceding decomposition, $L=L^{\op{nc}}\op{C}_L(L^{\op{nc}})$ with compact second factor. The uniqueness of the real Jordan decomposition therefore implies that the $\mathbb R$-split subgroup $A$ is contained in $L^{\op{nc}}$. Since $L$ is $\Theta$-invariant, so is $L^{\op{nc}}$, and hence $L^{\op{nc}}\in\mathscr H^*$. Therefore $H:=\op{N}_G(L^{\op{nc}})^\circ\in\mathscr H$.

The normalizer of a reductive subgroup $R<G$ is commensurable with $R\op{C}_G(R)$. Hence $H$ is commensurable with $L^{\op{nc}}\op{C}_G(L^{\op{nc}})$. Since $\op{C}_G(L^{\op{nc}})$ is compact, we have $H^{\op{nc}}=L^{\op{nc}}$. Moreover, $\op{C}_G(L^{\op{nc}})$ is $\Theta$-invariant and compact, and is therefore contained in $K$. It follows that
$$
H.o=L^{\op{nc}}.o=L.o.
$$
\end{proof}

\begin{lem}\label{inv}
    Let $H<G$ be a connected $\Theta$-invariant closed subgroup such that $\op{N}_G(H)\cap \exp \frak p\subset H$.  If $gH g^{-1}$ is $\Theta$-invariant for some $g\in G$, then there exists $h\in H$ such that
    $$k:=gh^{-1}\in K\quad\text{ and }\quad  gHg^{-1}= kHk^{-1}.$$
\end{lem}
\begin{proof} 
Let $\frak h$ be the Lie algebra of $H$. Since $H$ is connected and $\Theta$-invariant, it has the global Cartan decomposition
$$
H=(H\cap K)\exp(\frak h\cap\frak p).
$$
Together with the uniqueness of the Cartan decomposition in $G$, this gives
$$
H\cap\exp\frak p=\exp(\frak h\cap\frak p).
$$
Since $\Theta(g)H\Theta(g)^{-1}=gHg^{-1}$, we have $n_g:=g^{-1}\Theta(g)\in\op{N}_G(H)$. Moreover $\Theta(n_g)=n_g^{-1}$, and hence $n_g\in\exp\frak p$. By hypothesis, $n_g\in H$. Write $g=k\exp(y/2)$ with $k\in K$ and $y\in\frak p$. Then $n_g=\exp(-y)\in H\cap\exp\frak p=\exp(\frak h\cap\frak p)$. Since $\exp|_{\frak p}$ is injective, $y\in\frak h\cap\frak p$. Thus $h:=\exp(y/2)\in H$ and $k=gh^{-1}\in K$, as required.
\end{proof}

\begin{lem}\label{ngh}
    For every $H\in \mathscr H^*\cup \mathscr H$, we have
   $\op{N}_G(H) \cap \exp \frak p\subset H$. 
\end{lem}
\begin{proof} 
  Let $H\in \mathscr H^*\cup \mathscr H$.
  Set $H':=\op{N}_G(H)$
  and  $\frak h'=\op{Lie} (H')$. 
  Since $H$ is $\Theta$-invariant,
  so is $H'$. Hence $H'$ admits the Cartan decomposition
   $H'=(K\cap{H'})\exp(\frak p\cap\frak h')$, and 
 by the uniqueness of the Cartan decomposition,
$$
H'\cap\exp(\frak p)=\exp(\frak p\cap\frak h').
$$
Now for $H\in \mathscr H^*\cup \mathscr H$, $\op{C}_G(H)$
is a subgroup of $K$ and hence its Lie algebra is contained in $\frak k$. Therefore $$\frak p\cap \frak h'=\frak p \cap (\frak h \oplus\op{Lie}(\op{C}_G(H)) =\frak p\cap \frak h.$$
Thus
$$H'\cap\exp(\frak p) = \exp(\frak p\cap\frak h)\subset H.$$
\end{proof}

\begin{lem}\label{lem.e1}
    Let $L$ be a connected reductive subgroup of $G$ containing some $H\in\mathscr H^*$.
    Then $L$ is $\Theta$-invariant.
\end{lem}
\begin{proof} If $L=G$, there is nothing to prove. Suppose therefore that $L\ne G$. By \cite[Theorem~7.3]{Mos}, together with the conjugacy of Cartan involutions \cite[Chapter~III, Theorem~7.2]{He}, there exists $g\in G$ such that both $gHg^{-1} $ and $ gLg^{-1}$ are $\Theta$-invariant.  By applying Lemmas \ref{inv} and \ref{ngh} to $H$ and $gHg^{-1}$,  we obtain $h\in H$ such that $k:=gh^{-1}\in K$ and $gHg^{-1}=kHk^{-1}$.
    Since $g=kh\in kL$, we have $gLg^{-1}=kLk^{-1}$.
    Hence $kLk^{-1}$ is $\Theta$-invariant. This implies that $L$ is $\Theta$-invariant.
\end{proof}

\subsection*{Unimodular connected subgroups of $G$}
A closed subgroup of $G$ is unimodular if its left and right Haar measures coincide. A closed subgroup admitting a lattice is a unimodular subgroup \cite[Remark~1.9]{Rag}.
We will describe all unimodular closed subgroups of $G$ using the following lemma.
\begin{lem}\label{vpp}
    Let $V<N$ be a non-trivial subgroup.
    \begin{enumerate}
        \item  If $g\in G$ satisfies $gVg^{-1}\cap V\ne \{e\}$, then $g\in P$.
        \item The normalizer of $V$ in $G$ is contained in $P$.
 \end{enumerate}
\end{lem}
\begin{proof}
 By the Bruhat decomposition \cite[Chapter~IX, Theorem~1.4]{He},
$$
G=P\sqcup P\wo P.
$$
Suppose that $g=p_1\wo p_2\in P\wo P$ and choose a non-trivial element $u\in gVg^{-1}\cap V$. Since $N$ is normal in $P$, we have $p_2Vp_2^{-1}\subset N$. Hence
$$
p_1^{-1}up_1\in \wo N\wo^{-1}=N^+\subset P^+.
$$
On the other hand, $u\in V\subset P$ and $p_1\in P$, so $p_1^{-1}up_1\in P$. Thus $P\cap P^+=MA$ contains a non-trivial unipotent element, a contradiction. This proves (1), and (2) follows immediately from (1).
\end{proof}

\begin{lem}\label{pppp} Any connected closed unimodular subgroup $L$ of $G$ is either reductive or of the form $L= QV $ where $V$ is the unipotent radical of $L$ and $Q$ is a compact subgroup. 
\end{lem}
\begin{proof}
 Suppose $L$ is not reductive, ie, its unipotent radical $V$ is not trivial.
We have a Levi decomposition $L=QV$
where $Q$ is a reductive subgroup. Since $Q$ normalizes $V$, it is
contained in a parabolic subgroup by Lemma \ref{vpp}.
This implies  $Q^{\nc}$ should be trivial, since a parabolic subgroup in rank one cannot contain a noncompact simple Lie subgroup. Hence $Q$ is an almost direct product of a torus and compact simple Lie group. Since $L$ is unimodular, $Q$ cannot have an $\br$-split torus, as its conjugation action on $V$ would yield a non-trivial modular character of $L$. Therefore $Q$ must be compact. 
\end{proof}

\section{Totally geodesic submanifolds}\label{sec.submanifolds}
In this section, we characterize the totally geodesic subspaces of the symmetric space $X=G/K$ and describe their stabilizers in $G$. These subspaces correspond precisely to the orbits of subgroups belonging to the families $\mathscr H$ and $\mathscr H^*$
 introduced in Section \ref{sec.b}. We also explain how they descend to immersed totally geodesic submanifolds in the quotient manifold $\cM=\Ga\ba X$. Recall the basepoint $o=[K]\in X$.

A totally geodesic subspace of $X$ of dimension one is a complete geodesic, and it is of the form
$gA.o$ for some $g\in G$. Those of dimension at least two can be characterized as follows:
\begin{lem}\label{lem.PE0} For a connected smooth submanifold $Y$ of $X$ of dimension at least two, the following are equivalent to each other:
\begin{enumerate}
    \item $Y$
is totally geodesic;
\item $Y=gH.o$ for some $g\in G$ and $H\in\mathscr H$;
\item  $Y=gH.o$ for some $g\in G$ and $H\in\mathscr H^*$. 
\end{enumerate}
\end{lem}
\begin{proof}
If $H\in\mathscr H^*\cup\mathscr H$, then $H$ is $\Theta$-invariant and $\frak p^*:=\frak h\cap\frak p$ is a Lie triple system; ie,
$[[\frak p^*,\frak p^*],\frak p^*]\subset\frak p^*$.
Hence $H.o=\exp(\frak h\cap\frak p).o$ is totally geodesic by \cite[Chapter~IV, Theorem~7.2]{He}. This proves that (2) and (3) imply (1).

Conversely, suppose that $Y$ is totally geodesic. After translating by an element of $G$, we may assume that $o\in Y$. By \cite[Chapter~IV, Theorem~7.2]{He} (see also \cite[Proposition~2.6.1]{Eb}), there is a Lie triple system $\frak p^*\subset\frak p$ such that
$Y=\exp(\frak p^*).o$.
Put $\frak k^*:=[\frak p^*,\frak p^*]$ and $\frak l^*:=\frak k^*\oplus\frak p^*$. The proof of \cite[Chapter~IV, Theorem~7.2]{He} shows that the connected Lie subgroup $L^*<G$ with Lie algebra $\frak l^*$ preserves $Y$ and acts transitively on it.

The map $\frak p\to X$, $v\mapsto\exp(v).o$, is a diffeomorphism by \cite[Chapter~VI, Theorem~1.1(iii)]{He}; hence $Y=\exp(\frak p^*).o$ is closed. Therefore its stabilizer $G_Y$ is closed. Set $L:=G_Y^\circ$. Then $L$ is a connected closed subgroup, $L^*<L$, and $L.o=Y$. Since
$\Theta(\exp(v).o)=\exp(-v).o$ for all $v\in\frak p$,
and $\frak p^*$ is a vector subspace of $\frak p$, $\Theta$ preserves $Y$. Hence
$\Theta(G_Y)=G_Y$, and therefore $\Theta(L)=L$. Thus $L$ is reductive.

We claim that $L^{\op{nc}}\ne\{e\}$. Otherwise $[L,L]$ is compact and, by $\Theta$-invariance, is contained in $K$. It follows that $\frak l\cap\frak p$ is contained in the center of $\frak l$ and is therefore abelian. Since $G$ has real rank one,
$\dim Y=\dim(\frak l\cap\frak p)\le 1$,
contrary to the hypothesis. By the conjugacy of maximal flats through $o$ \cite[Chapter~V, Theorem~6.2(i)]{He}, there exists $k\in K$ such that $Q:=k^{-1}Lk$ contains $A$. Set
$$
H:=\op{N}_G(Q^{\op{nc}})^\circ.
$$
By Lemma~\ref{lem.h1}, $H\in\mathscr H$, $H^{\op{nc}}=Q^{\op{nc}}$, and $H.o=Q.o=Q^{\op{nc}}.o$. Hence
$$
Y=kH.o=kH^{\op{nc}}.o,
$$
which proves (2) and (3).
\end{proof}

\begin{lem}\label{gy} \begin{enumerate}
    \item If $Y=gH.o$ for some $H\in \mathscr H \cup\{A\}$, then $G_Y^\circ =g\op{N}_G(H)^\circ g^{-1}$. In particular, if $H\in \mathscr H$, then $G_Y^\circ= g H g^{-1}$. \item  If $G_{Y_1}^\circ =G_{Y_2}^\circ $ for totally geodesic subspaces $Y_1, Y_2$ of $X$, then $Y_1=Y_2$.
\end{enumerate}
\end{lem}
\begin{proof}
For (1), we may assume that $g=e$. Let $\frak h=\op{Lie}(H)$ and $\frak p_H:=\frak h\cap\frak p=T_oY$. If $g'\in G_Y$, then $g'.o\in H.o$, so $g'=hk$ for some $h\in H$ and $k\in K$. Since $kY=Y$ and $k.o=o$, we have $\op{Ad}(k)\frak p_H=\frak p_H$.
If $H=A$, then $\frak p_H=\frak a$, and hence $k\in\op{N}_K(A)$. Thus $G_Y=A\op{N}_K(A)=\op{N}_G(A)$.

Now suppose that $H\in\mathscr H$ and set $S:=H^{\op{nc}}$, with Lie algebra $\frak s$. By Lemma~\ref{lem.h1}, $H.o=S.o$, and hence $\frak p_H=\frak s\cap\frak p$. Since $\frak s$ is semisimple and $\theta$-invariant,
$$
\frak s=(\frak s\cap\frak k)\oplus\frak p_H
\quad\text{and}\quad
\frak s\cap\frak k=[\frak p_H,\frak p_H].
$$
It follows that $\op{Ad}(k)\frak s=\frak s$, so $k$ normalizes $S$. Since $H=\op{N}_G(S)^\circ$, the subgroup $H$ is also normalized by $k$. Therefore $G_Y\subset\op{N}_G(H)$.
Conversely, since $H$ is reductive, $\op{N}_G(H)$ is commensurable with $H\op{C}_G(H)$. Moreover,
$$
\op{C}_G(H)\subset\op{C}_G(S)\subset K.
$$
Thus $\op{N}_G(H)^\circ$ stabilizes $H.o$, and hence $\op{N}_G(H)^\circ\subset G_Y^\circ$. This proves
$G_Y^\circ=\op{N}_G(H)^\circ.$
Since $S=H^{\op{nc}}$ is the unique maximal connected normal
semisimple subgroup of $H$ with no compact factors, every element
normalizing $H$ also normalizes $S$. Hence
$$
\op{N}_G(H)^\circ\subset\op{N}_G(S)^\circ=H.
$$
The reverse inclusion is immediate, and therefore $G_Y^\circ=H$.

For (2), write $Y_i=g_iH_i.o$ with $H_i\in\mathscr H\cup\{A\}$. If exactly one of $H_1,H_2$ is $A$, the groups $G_{Y_1}^\circ$ and $G_{Y_2}^\circ$ cannot be equal: $\op{N}_G(A)^\circ$ has no non-compact semisimple normal subgroup, whereas every group in $\mathscr H$ does.  If $H_1=H_2=A$, then, by (1), $g_2^{-1}g_1$ normalizes $\op{N}_G(A)^\circ=\op{C}_G(A)^\circ$.
Conjugation by $g_2^{-1}g_1$ therefore preserves the center of
$\op{C}_G(A)^\circ$ and its unique maximal connected
$\mathbb R$-split subgroup, namely $A$. Hence
$g_2^{-1}g_1\in\op{N}_G(A)$, and consequently $Y_1=Y_2$.

It remains to consider $H_1,H_2\in\mathscr H$. In this case, (1) gives
$g_1H_1g_1^{-1}=g_2H_2g_2^{-1}$.
Set $q:=g_2^{-1}g_1$. By Lemmas~\ref{inv} and~\ref{ngh}, there are $h\in H_1$ and $k\in K$ such that $q=kh$ and $kH_1k^{-1}=H_2$. Therefore
$$
Y_1=g_2kH_1.o=g_2H_2.o=Y_2.
$$
\end{proof}

A totally geodesic subspace $Y\subset X$ of dimension at least two equipped with the metric induced from $(X,\sfd)$, is itself a rank-one symmetric space of non-compact type.
Moreover, its boundary at infinity $\partial_\infty Y$ can be naturally identified with a subset of $\partial_\infty X$.
Indeed, if $Y=gH.o$ for $H\in \mathscr H$, then, since $A\subset H$,
we have $$\partial_\infty Y= gHP/P\subset \partial_\infty X=G/P .$$
\begin{lem}\label{yb}
   We have $\hull (\partial_\infty Y)=Y$. 
\end{lem}
\begin{proof} We use the fact that between any two distinct points of $\partial_\infty X$, there is a unique geodesic with those endpoints. Let $\ell\subset X$ be a geodesic with both end points
in $\partial_\infty Y$. Since $Y$ itself is a rank-one symmetric space, inside $Y$, there exists a unique geodesic $\ell_Y$ with the same points. By the uniqueness, $\ell=\ell_Y\subset Y$, proving
the inclusion $\hull (\partial_\infty Y)\subset Y$. Conversely,
for every $y\in Y$, there exists a complete geodesic through $y$ with end points in $\partial_\infty Y$. Hence $Y\subset \hull (\partial_\infty Y)$, and the equality follows.
\end{proof}

\subsection*{Properly immersed totally geodesic submanifolds}
Let $\Ga<G$ be as in the introduction, and consider the associated locally symmetric space
$$
\cM=\Ga\ba X.
$$
Let  $\mathsf p: X\to \Ga\ba X$ denote the quotient map.
For a totally geodesic subspace $Y$ of $X$,
the restriction $\sfp|_Y: Y\to\Ga\ba X$ factors through the covering map $Y\to   (G_Y\cap \Ga)\ba Y$, and hence induces an immersion
$$
\iota:  (G_Y\cap \Ga) \ba Y \to \cal M.
$$
Its image $\cal N=\Ga\ba\Ga Y$ is a totally geodesic (immersed) submanifold of $\cM$.
If the map $\iota$ is proper, then $\cal N$ is a \textit{properly immersed totally geodesic submanifold} of $\cal M$.

We will use the following standard properness criterion for homogeneous orbits and include its proof for completeness\footnote{The proof of the lemma appears in the arXiv version of \cite{OS1}, but was omitted from the published version; we reproduce it here for the reader's convenience.}.
\begin{Lem}\label{properr}
Let $G$ be a locally compact, second countable topological group and $H<G$ a closed subgroup. For every closed subgroup $\Ga <G$, the  canonical projection $(H\cap \Ga) \ba H \to \Ga\ba G$ is proper if and only if
$\Ga H$ is closed in $G$.
\end{Lem}
\begin{proof} Suppose that $\Ga H$ is closed in $G$.
Let $x_0=[e]\in\Ga \backslash G$ and $Z=x_0 H$.
Then $Z$ is closed in $\Ga \backslash G$; in particular, $Z$ is locally compact and second countable and $H$ acts transitively on $Z$. We first claim, by Baire's category theorem, that the map
$h\mapsto x_0h$ from $H$ to $Z$ is open.
Indeed, let $\Omega$ be a neighborhood of $e$ in $H$ and fix $h\in H$.
We wish to show that $x_0h\Omega$ contains a neighborhood of $xh$.
Choose an open neighborhood $\Omega_1$ of $e$
with compact closure and the closure of ${\Omega_1 \Omega_1^{-1}}$ is contained in $ \Omega$.
There exists a sequence $\{g_i\}\subset H$ such that $H=h\cup (\cup_{i=1}^\infty\Omega_1)g_i$.
Hence $Z=\cup_{i=1}^\infty x_0h\cl{\Omega}_1$.
By Baire's category theorem,  one of the sets $x_0h\cl{\Omega}_1g_i$ has nonempty interior, say around $x_0h\omega g_i$ for some $\omega\in \cl{\Omega}_1$.
Thus $x_0h\cl{\Omega}_1g_i(g_i^{-1} \omega^{-1})$ contains a neighborhood of $x_0h$, proving that the map $h\mapsto x_0 h$ is open.
Therefore the map $$(H\cap \Ga) \backslash H \to Z ,\quad (H\cap \Ga) h\mapsto x_0 h$$ is continuous, open, and one to one, hence a homeomorphism.

Now let $C\subset \Ga \backslash G$ be compact. 
Then its inverse image in $(H\cap \Ga) \backslash H$ coincides with the inverse image of $C\cap Z$ in $(H\cap \Ga) \backslash H$, which is compact
by the above homeomorphism.
This proves that the projection map is proper.
The converse is straightforward.
\end{proof}

\begin{lem}  \label{OS}
     Let $\Ga<G$ be a discrete subgroup and let
    $\cal N=xH.o$ be a totally geodesic submanifold for some $x\in\Ga\ba G$ and $H\in\mathscr H$.
   Then the  following are equivalent:
   \begin{enumerate}
       \item  $\cal N$ is properly immersed in $\cal M$;
       \item $\cal N$ is closed in $\cal M$;
\item  $xH$ is closed in $\Ga\ba G$.   \end{enumerate}
\end{lem}

\begin{proof} We may assume without loss of generality that $x=[e]$, so that $\cal N=\Ga\ba \Ga Y$ where $Y=H.o$.  We first show that $(3) \Rightarrow (1)$.
By hypothesis, the image of $ (\Ga\cap H) \ba H\to \Ga\ba G$ is closed.
Hence by Lemma \ref{properr}, the inclusion map
$j: (\Ga\cap H)\ba H\to \Ga \ba G$ is proper.
Since $p: \Ga\ba G\to \cM=\Ga\ba G .o$, $x\mapsto x.o$ is a proper map, 
the composition $p\circ j: (\Ga\cap H)\ba H\to \cM$ is proper. Since $p\circ j$ factors through  the map $j': (\Ga\cap H)\ba Y\to \cM$ induced by the inclusion $Y\subset X$,
$j'$ is a proper map.   Since $j'$ factors through the map $\iota: (\Ga\cap G_Y)\ba Y\to \cM $, this implies that $\iota$ is proper; $\cN$ is properly immersed, as desired.
The implication  $(1)\Rightarrow (2)$ is immediate.

We now show $(2)\Rightarrow (3)$. Since $[G_Y:H]<\infty$ by
Lemma \ref{gy}, it suffices to show that $\Ga G_Y$ is closed in $ G$.
Consider the set $\cal S=\{gY:g\in G\}=G Y$ with the topology given by the identification $\cal S\simeq G/G_Y$. Suppose that $\cal N$ is closed in $\cM$, that is, $\Ga Y$ is closed in $X$. This implies $\Ga Y$ is closed in $\cal S$; if $\ga_i Y\to gY$, then,
since $\Ga Y$ is closed in $X$, we have $gY\subset \Ga Y=\cup_{\ga\in \Ga} \ga Y$ as a subset of $X$. This implies that  $gY=\ga Y$ for some $\ga\in\Ga$; otherwise, $gY\cap \ga Y$ is a nowhere dense subset of $gY$ for all $\ga \in \Ga$, contradicting the Baire Category theorem.

Now to show that $\Ga G_Y$ is closed in $G$, suppose that $\ga_i h_i $ converges to $ g\in G$ for some sequences $\ga_i\in \Ga$ and $h_i\in G_Y$. This implies that $\ga_i Y$ converges to $gY$ in $\cal S$. By the closedness of $\Ga Y$ in $\cal S$, we then have $gY=\ga Y$ for some $\ga \in \Ga$. Therefore $g=\ga h$ for some $h\in G_Y$.  This implies that $\Ga G_Y$ is closed in $G$, completing the proof.
\end{proof}

In Lemma \ref{OS}, it is important that $H$ is chosen from the collection $\mathscr H$ so that $H$ is equal to the connected component of its normalizer, since $xH^{\nc}$ may not be closed even if $xH^{nc}.o=xH.o$.

\section{Totally geodesic submanifolds contained in the convex core}\label{sec.gf}
Let $\Ga<G$ be a discrete subgroup, and consider the associated locally symmetric space $\cal M=\Ga\ba X$.  
In this section, we focus on totally geodesic submanifolds contained in the convex core of $\cM$.

We begin by recalling several basic definitions.
The \textit{limit set} of $\Ga$, denoted by $\La=\La(\Ga)$, is the set of accumulation points of $\Ga(o)$ in $\partial_\infty\xx$ within the compactification $X\cup \partial_\infty\xx$.
We assume that $\Ga$ is non-elementary, that is, $\# \La\ge 3$. Then $\La$ is the unique $\Ga$-minimal subset of $\partial_\infty\xx$. 
The \textit{convex hull} of $\La$, denoted by $\text{hull}(\La)$, is the smallest convex subset of $\xx$ containing all geodesics\footnote{by geodesics, we mean complete geodesics} in $\xx$ with endpoints in $\La$.
The \textit{convex core} of $\cal M$ is defined as the quotient manifold
$$
\cm=\Ga\ba \text{hull}(\Lambda).
$$

For $p\in \cal M$, the \textit{injectivity radius} of $p$ is the supremum of $r>0$ such that the ball $B(\tilde p, r)$ injects into $\cal M$, where $\tilde p\in X$ satisfies $p=\Gamma\ba \Gamma \tilde p$. We  denote it by $\inj(p)$. For $\e>0$, define the $\e$-thin and $\e$-thick parts of $\cal M$ by
$$
\cal M_\e=\{p\in\cal M:\inj(p)<\e\} \;\; \text{ and } \;\; \cal M_{\geq\e}=\cal M-\cal M_\e .
$$

\subsection*{Horospheres and horoballs} 
If  $\xi=gP\in \partial_\infty X $ for $g\in G$, then a horosphere of $X$ based at $\xi$ is of the form $g a N.o$ for some $a\in A$. Similarly, a horoball of $X$ based at $\xi $ is of the form $g A_{[T, \infty]} N.o$ where $A_{[T, \infty)}=\{a_t: t\ge T\}$. A horosphere (resp horoball) in $\cal M$ is then the image of a horosphere (resp horoball) in $X$ under the quotient map $X\to \Ga\ba X$.

\subsection*{Geometrically finite manifolds}
We now recall the notion of geometrically finite manifolds.
\begin{Def}\label{def.geometricfiniteness} 
We say that $\cal M$, or equivalently $\Ga$, is geometrically finite, if the unit neighborhood of $\cm$ has finite volume. 
\end{Def}

\begin{thm}\cite{Bow2}\label{lem.f2}
The manifold $\cM$ is geometrically finite if
  and only if $\cal M_{\geq\e}\cap \core (\cM)$ is compact for all sufficiently small $\e>0$.
\end{thm}

A limit point $\xi\in \La $ is \textit{parabolic} if its stabilizer 
$$
\Ga_\xi:=\{g\in \Ga: g\xi=\xi\}
$$ contains a parabolic element. We denote by $\La_{\op{p}}=\La_{\op{p}}(\Ga) $ the set of all parabolic limit points of $\Ga$. 

We use the standard cusp description of geometrically finite rank-one manifolds, including finiteness of the parabolic orbits and the disjoint horoball decomposition below; see \cite[Section~4]{Bow2}.
For geometrically finite $\cal M$, there exists a finite set of $\Ga$-representatives $\xi_i\in \La_{\op{p}} $ $(1\leq i\leq \ell)$ such that 
\begin{equation}\label{eq.par}
\La_{\op{p}} =\Ga\xi_1\cup\cdots\cup\Ga\xi_\ell . 
\end{equation}
 Fix $g_i\in G$
so that $\xi_i=g_iP\in \La_{\op{p}}$.
For $T>0$, define the horoball based at $\xi_i$ of depth $T$ by
 \be\label{h} \tilde {\cal H}_{\xi_i, T} = g_iNA_{[T, \infty)}. o \subset X\quad  \text{ and } \quad 
 \cal H_{\xi_i, T}  =\Ga \ba \Ga \tilde{\cal H}_{\xi_i, T} \subset \cal M.\ee
 For all sufficiently small $\e>0$, there exist $T_1, \cdots, T_\ell >0$ such that for each $1\le i\le \ell$,
the collection $\{\gamma \tilde {\cal H}_{\xi_i, T_i}:\gamma\in \Ga\}$ consists of disjoint horoballs, and 
the following disjoint decomposition holds:
\begin{equation}\label{eq.tt}
\cal M=\cal M_{\geq\e }\cup \cal H_{\xi_1, T_1} \cup\cdots\cup\cal H_{\xi_\ell, T_\ell}.  
\end{equation}

We record a simple observation for later use.
\begin{lem}\label{complete}  For each $1\le i \le\ell$, the horoball 
$\cal H_{\xi_i, T_i}$
contains no complete geodesic.
    \end{lem}\begin{proof} 
        Suppose that $\cal H_{\xi_i, T_i}$ contains a complete geodesic.  By the disjointness property of the horoballs $\{\gamma \tilde {\cal H}_{\xi_i, T_i}:\ga\in \Ga\}$
      some translate $\gamma \tilde {\cal H}_{\xi_i, T_i}$ would then contain a complete geodesic of $X$.  However, every complete geodesic in $X$ 
         connects two distinct points of $\partial_\infty X$, a contradiction.
    \end{proof}

\subsection*{Totally geodesic submanifolds contained in the convex core}
For each closed subgroup $H$ of $G$ with Lie algebra $\frak h$, the restriction $\langle \cdot, \cdot\rangle|_{\frak h}$ defines an inner product on $\frak h$. 
Left translation within $H$ produces a left-invariant sub-Riemannian metric on $H$, which in turn determines a unique left-invariant volume form $\op{vol}_H$, or equivalently a left-invariant Haar measure $\mu_H$, on $H$.
If $H$ is unimodular, $\mu_H$ is also right $H$-invariant.

For an orbit $xH\subset \Ga \ba G$ with $H\in \mathscr H$,
the measure $\mu_H$ induces a unique
right $H$-invariant measure on $\op{Stab}_H(x) \ba H$, which we again denote by $\mu_H$ by abuse of notation.
We write $$\vol(xH) =\mu_H( \op{Stab}_H(x) \ba H) $$ and set
$$
\vol(xH.o) =\frac{\vol(xH)}{\mu_{H\cap K}(H\cap K)}.
$$

We define the set\footnote{The notation $\RFM$ was introduced in \cite{MMO1} for the case $G=\op{SO}(3,1)^\circ$, where it denotes the \textit{renormalized frame bundle}.
When $G=\op{SO}(n,1)^\circ$, the quotient $\Ga\ba G$ can be identified with the  frame bundle of $\cal M=\Ga\ba X$.
This identification does not extend to general rank-one spaces, so the term "frame bundle" is slightly abusive. Nevertheless, we retain the notation for consistency.
}
\begin{equation}\label{eq.xe0}
\RFM=\{[g]\in \Ga\ba G: gP,\;g\wo P \in\La\}.   
\end{equation}
For $g\in G$, the points $gP$ and $g\wo P$ are precisely the two endpoints of the
geodesic $gA.o$.
Note that if one of the end points of $gA.o$ lies outside the limit set, $gA.o$ is unbounded. Therefore, every bounded $A$-orbit is contained in $\RFM$; 

Note that $\RFM$ is right-invariant under the subgroup
$P\cap\wo P\wo^{-1}= \op{C}_G(A)$ and the centralizer $\op{C}_G(A)$ is equal to $A\op{C}_K(A)$.
It was shown in \cite{MMO} that the periodic $A\op{C}_K(A)$ orbits are equidistributed in $\RFM$ with respect to the Bowen-Margulis-Sullivan measure; in particular, they are dense in $\RFM$. Since such periodic orbits are bounded,
their closure is contained in the smallest closed subset of $\Ga\ba G$ containing all bounded $A$-orbits, while the previous paragraph shows that every bounded $A$-orbit lies in $\RFM$.
Therefore we obtain:
\begin{prop}\label{rfm}
    For geometrically finite $\cM$, the set $\RFM$ coincides with the smallest closed subset of $\Ga\ba G$ containing all bounded $A$-orbits in $\Ga\ba G$.
\end{prop}

The image of $\RFM$ under the basepoint projection $\Ga\ba G\to \cal M$ is contained in $\cm$, though the two sets need not coincide in general.

The next lemma characterizes when a totally geodesic submanifold of $\cM$ lies in the convex core in terms of its associated 
$H$-orbit in $\RFM$.
\begin{lemma}\label{vol} \label{lem.q} Let $\cal N$ be a totally geodesic submanifold of $\cal M$ of dimension at least two.
Let $\cal N=x H.o$ for some $H\in\mathscr H$.
    Then $$(1)\Rightarrow (2)\Leftrightarrow (3) \Leftrightarrow (4)$$ 
where
    \begin{enumerate}
        \item $\mu_H(xH)<\infty $, equivalently, $\op{vol}(\cal N)<\infty .$
    \item $\cal N \subset \cm$.
        \item[(3)] $xH\subset\RFM$.
        \item[(4)] $xH^{\nc}\subset \RFM$.
    \end{enumerate}
\end{lemma}
\begin{proof}  Let $g\in G$ represent $x=\Ga g$ and let $\pi: \Ga\ba G\to \Ga\ba X$ denote the natural projection $x\mapsto x.o$. The implication $(3) \Rightarrow (2)$ holds since $\pi(\RFM)\subset\cm$.
To prove $(2)\Rightarrow (3)$, suppose  $xH.o\subset \cm$.
Then $gH. o\subset \hull (\La)$. 
For $h\in H$, consider
$$(gh)^+:=ghP,\qquad (gh)^-:=gh\wo P\in\partial_\infty X.$$
The points $(gh)^{\pm}$ are represented by
the geodesic rays $\{gha_t.o: t\ge 0\}$
and $\{gha_t.o: t\le 0\}$ respectively.
Recall the fact that
$\overline{\hull (\La)}\cap \partial_\infty X$
coincides with $\La$,
where
 $\overline{\hull (\La)}$ denotes the closure in the compactification $X\cup \partial_\infty X$.
Since $ghA.o\subset \hull (\La)$, it follows that $(gh)^{\pm}\in \La$, that is, $gh\in \RFM$, proving $(2)\Leftrightarrow (3)$.
Since $xH^{\nc}.o=\cal N$, this also implies that $(2)\Leftrightarrow  (4)$.

       We now prove $(1)\Rightarrow (3)$.
    Suppose that the orbit $xH$ is not contained in $\RFM$.
    By replacing $x$ by $xh$ for some $h\in H$, we may assume that
    one of $gP$ or $g\wo P$ does not belong to the limit set.
Suppose $\xi=gP\not\in \La$.
    Since $\Ga$ acts properly discontinuously on $\partial_\infty X-\La$, there exists a neighborhood $\cal O$ of $\xi$ in the compactification $X\cup \partial_\infty X$ such that $\ga \cal O \cap \cal O=\emptyset$ for all $\ga\in \Ga-\{e\}$.
    Consider the geodesic ray $t\mapsto ga_t.o$, $t\ge 0$.
    For all sufficiently large $t$, we have $ga_t.o\in \cal O$, and the injectivity radius $r_t:=\inj_{\Ga\ba G}(xa_t)$ tends to $\infty$ as $t\to\infty$.
    Since $A\subset H$, the ball $B_{r_t}^H(a_t):=\{h\in H:  \mathsf d(a_t, h)<r_t\}$ injects under $h\mapsto xh$ into $xH=\Stab_H(x)\ba H$.
    By right invariance of Haar measure, $\mu_H(B_{r_t}^H(a_t))=\mu_H(B_{r_t}^H(e))$, and hence
    $$
    \mu_H (xH) \ge \limsup_{t\to \infty} \mu_H(B_{r_t} ^H(e)) =\mu_H(H)=\infty.
    $$
    Hence $\mu_H(xH)=\infty$. The case $\xi=g\wo P\not\in \La$ can be proved similarly by sending $t\to -\infty$. 
    \end{proof}
We also observe:
\begin{lem}\label{lem.boundedincore} Let $\cN$ be as in Lemma \ref{vol}. If $\cN$ is bounded,
then $\cN\subset \core (\cM)$.
\end{lem}
\begin{proof}
Suppose that $\cN=xH.o$ is bounded, or equivalently, $xH$ is bounded. For every $h\in H$, since $xhA\subset xH$ is bounded, we have $xh\in \RFM$.
    Hence $xH\subset \RFM$, which implies that $\cN\subset \core (\cM)$.
\end{proof}

\begin{rmk}\label{rmk.horospheres} \rm We remark that an analogue of Lemma \ref{vol}
   does not hold for horospheres. There are maximal horospheres
   $xN.o$ contained in the convex core of $\cal M$ but $xN k$ does not need to be contained in $\RFM$ for every $k\in K$. For example, consider a Kleinian group
   $\G<\PSL_2(\c)$ such that $\infty$, the fixed point of $N$, is a rank two parabolic limit point of $\Ga$  and $\La \ne \hat\c$. Then $[e]a_t N k \subset \Ga\ba G$ is not contained in $\RFM$ for every diagonal $a_t\in A$ and $k\in K$. But 
for all sufficiently large $t$, $[e]a_t N .o$ is contained in the convex core of $\Ga\ba \bH^3$.
\end{rmk}

 \section{Orbit closure classification inside $\RFM$}\label{sec.orbitclosures}
 In this section, we classify the closures of orbits of connected closed subgroups generated by unipotent elements contained in the renormalized frame bundle  $\RFM$.
 This orbit-closure theorem forms the dynamical backbone of the rigidity and finiteness results proved later. Its proof relies on Ratner's measure classification theorem and the avoidance theorem of Dani-Margulis, adapted to the geometrically finite, possibly infinite-volume, setting.

\begin{theorem}\label{Theoremmax}
Let $\Ga<G$ be geometrically finite and $W<G$ a connected closed subgroup generated by one-parameter unipotent subgroups. Suppose that $xW\subset \RFM$ for $x\in \Ga\ba G$.
Then there exists a connected  Lie subgroup $L$ containing $W$ such that
$$\overline{xW} =xL$$
and $ \op{Stab}_L(x)$ is a lattice in $L$. Moreover, $L$ is either reductive or a compact extension of a connected unipotent subgroup (as given by Lemma \ref{pppp}).
\end{theorem}

\subsection*{Three theorems for general $\Gamma$ and $G$}
The rest of the section is devoted to proving Theorem \ref{Theoremmax}. We begin by recalling several foundational results on unipotent dynamics:
\begin{thm}[Ratner \cite{Ra}]\label{ra} Let $G$ be a connected linear Lie group and $\Ga<G$ be a discrete subgroup. Let $U$ be a one-parameter unipotent subgroup of $G$.
   Any $U$-invariant, ergodic probability measure $\nu$ on $\Ga\ba G$ is an $L$-invariant measure supported on a closed orbit $xL\subset \Ga \ba G$ for some $x\in \Ga\ba G$ and some connected closed subgroup $L<G$ containing $U$.
\end{thm}
This fundamental result describes all ergodic invariant probability measures for one-parameter unipotent flows.
The next theorem, due to Ratner as well, asserts that for a fixed base point, only countably many homogeneous subspaces can arise as supports of such invariant measures.
\begin{thm}[Ratner {\cite[Theorem 5]{RaGF}}{; see also Dani-Margulis \cite[Proposition~2.1]{DM}}] \label{ragf} Let $G$ be a connected linear Lie group and $\Ga<G$ be a discrete subgroup. Let $x\in \Ga\ba G$.
Let $\cal A_x$ denote the set of all closed connected subgroups $L<G$
such that $xL$ is closed and has an $L$-invariant probability measure $\mu_L$ and there is a one-parameter unipotent subgroup $U<L$ acting ergodically on $(xL, \mu_L)$. Then $\cal A_x$ is countable.
    
\end{thm}

To formulate the avoidance theorem of Dani--Margulis, we introduce the notion of the singular set associated to a subgroup.
\begin{Def} \label{s} For a connected closed subgroup $W<G$, define the {\it singular set} with respect to $W$:
$$
\mathscr{S}(W)=\left\{x\in\Ga\ba G:
\begin{array}{c}
     \text{there exists a closed connected subgroup $ L\subsetneq G$ s.t.}\\
     \text{$W<L$ and $xL$ admits a finite $L$-invariant measure} 
\end{array}
\right\}.
$$
\end{Def}
When an orbit $xL\subset \Ga\ba G$ admits a finite $L$-invariant measure, $xL$ is automatically a closed subset of $\Ga\ba G$ \cite[Theorem 1.13]{Rag}.

 The avoidance theorem of Dani--Margulis asserts that the orbit of a generic point under a one-parameter unipotent subgroup  spends a uniformly small proportion of time near the singular set.
This quantitative form of avoidance will later ensure that the measures obtained from orbit averages do not concentrate on lower-dimensional homogeneous subsets.
\begin{thm}[Dani-Margulis {\cite[Theorem~1]{DM}}]\label{lem.DM2}  Let $G$ be a connected linear Lie group and $\Ga<G$ be a discrete subgroup.
Let $W$ be a connected closed subgroup of $G$ which is generated by unipotent elements in it. Let $\mathcal K \subset \Ga\ba G$ be a compact subset disjoint from $\mathscr{S}(W)$.
     For every $\e>0$, there exists a neighborhood $\Omega$ of $\mathscr{S}(W)$ such that for every one-parameter unipotent subgroup $U=\{u_s:s\in \br \}$ of $G$, $x\in \mathcal K$, and $T>0$,
    $$
    \op{Leb}\{s\in [0,T]:xu_s\in\Omega\})\leq \epsilon T.  
    $$
    where $\op{Leb}$ denotes the Lebesgue measure of $\br$.
\end{thm}

\subsection*{Orbit closures inside $\RFM$ for geometrically finite $\cM$}
We now return to the setting of a geometrically finite manifold $\cM=\Gamma\backslash X$.
Denote by $$\pi:\Ga\ba G\to \cal M$$ the projection $x\mapsto x.o$.
To apply the general results recalled above, we need a quantitative non-divergence estimate for unipotent flows: unipotent orbits spend most of its time in the thick part of $\cM$.

\begin{thm}[{\cite[Corollary 5.5]{BO1}}, {\cite[Theorem 1.1]{BZ}}] \label{lem.DM1} Let $\Ga<G$ be a discrete subgroup.
    For every $\e>0$ and a compact set $\mathcal K\subset\Ga\ba G$, there exists a constant $\eta >0$ depending on $\mathcal K$ such that for every one-parameter unipotent subgroup $U=\{u_s:s\in \br\}$, $x\in \mathcal K$ and $T>0$,
    $$
  \op{Leb}\{s\in[0,T]: \pi(xu_s)\in \cM_\eta \}\le \e T 
    $$ where $\cal M_\eta$ denotes the $\eta$-thin part of $\cM=\Ga\ba X$ as in Section \ref{sec.gf}.
\end{thm}

The following proposition is the key step where the geometric finiteness assumption plays a decisive role. 
It guarantees that, within the renormalized frame bundle $\RFM$,
 every limiting measure arising from time-averages of unipotent orbits remains supported inside $\RFM$.

Denote by $\cal P(\Ga\ba G\cup\{\infty\})$ the space of probability measures on the one-point compactification of $\Ga \ba G$, equipped with the weak-$^*$ topology. This is a compact metrizable space.

\begin{prop}[Non-divergence within $\RFM$] \label{ND}  Let $\Ga<G$ be a geometrically finite subgroup and $\cM=\Ga\ba X$.
  Let $\mathcal K\subset\Ga\ba G$ be a compact subset, and let  $U_i=\{u_{i,s}:s\in \br \}$ be a sequence of one-parameter unipotent subgroups of $G$. Let  $x_i\in \mathcal K$ be a sequence  such that $x_iU_i\subset \RFM$ for each $i\ge 1$.
Then for every sequence $T_i\to \infty$, every weak-$^*$ limit of the sequence
$$\nu_{T_i}:=\frac{1}{T_i}\int_0^{T_i} \delta_{x_i u_{i,s}} \, ds $$ in $\cal P(\Ga\ba G\cup\{\infty\})$, as $i\to \infty$, is a probability measure supported on $\RFM$.
\end{prop}
\begin{proof} Fix $\e>0$ and a compact subset $\cal K\subset \Ga\ba G$.
  Let $\eta>0$ be as in Theorem \ref{lem.DM1}.  Then
  for every $x_i\in \cal K$ such that $x_iU_i\subset \RFM$,
  we have $\pi(x_iU_i)\subset \core (\cM )$. Hence Theorem \ref{lem.DM1} implies that for all $i\ge 1$ and $T>0$,
\be\label{al}  \op{Leb}\{s\in[0,T]: \pi(x_iu_{i,s})\in \core(\cM)  -\cM_\eta \}\ge (1-\e) T .\ee 
   Since $\cM$ is geometrically finite, the set
   $\core(\cM)  -\cM_\eta$ is compact. Hence $\cal C:=\{z\in \Ga \ba G: \pi(z)\in \core(\cM)  -\cM_\eta\}$ is a compact subset of $\Ga\ba G$.
By \eqref{al}, $$\nu_{T_i}( \cal C)\ge (1-\e )\quad\text{for all $i$.}$$
Therefore, for every weak-$^*$ limit $\nu$ of the sequence $\nu_{T_i}$, we have $\nu(\cal C)\ge 1-\e$. As $\e>0$ is arbitrary, it follows that $\nu(\Ga\ba G)=1$. Since $\text{supp} (\nu_{T_i})\subset \RFM$, we conclude that $\nu$ is  supported on $\RFM$.
\end{proof}

The next proposition is a key application of the avoidance and measure-classification theorems. It shows that, apart from points in the singular set, a unipotent orbit contained in $\RFM$ is already dense in $\Ga\ba G$ which in turn forces $\Ga$ to be a lattice.

\begin{prop}\label{prop.iso} Let $\Ga<G$ be  a geometrically finite  subgroup. Let $U$ be a one-parameter unipotent subgroup of $G$. 
    Suppose $xU\subset \RFM$ for $x\in \Ga\ba G$. Then $$\text{either }\quad  x\in\mathscr{S}(U)\quad \text{ or }\quad  \overline{xU}=\Ga\ba G.$$
    In the second case, $\Ga$ is a lattice in $G$.
\end{prop}
\begin{proof}  Suppose that $x\not\in\mathscr{S}(U)$. Let $U=\{u_s:s\in \br\}$ and assume $xU\subset \RFM$.
    For $T>0$, let $$\nu_T=\frac{1}{T}\int_0^T \delta_{x u_s}\, ds$$ be the probability measure on $\Ga\ba G$. 
    Let $\nu$ be a weak-$^*$ limit of a sequence
    $\nu_i:=\nu_{T_i}$  for some $T_i\to \infty$  in the one-point compactification of $\Ga\ba G$. Since $xU\subset \RFM$,  Proposition \ref{ND} implies that
    $$\nu(\Ga\ba G)=1.$$
On the other hand, since $x\not\in\mathscr{S}(U)$, 
Theorem \ref{lem.DM2} implies that
for every $\e>0$, there exists a neighborhood $\Omega $ of
$\mathscr{S}(U)$ such that for all $i$,
$$ \nu_i(\Omega) \le \e .$$
As $\e>0$ is arbitrary, it follows that
\be\label{si} \nu(\mathscr{S}(U))=0.\ee 

Since $\nu$ is a $U$-invariant probability measure on $\Ga\ba G$, it admits
 a $U$-ergodic decomposition $$\nu=\int_{\alpha\in \Ga\ba G} \nu_\al\,d\nu(\alpha) $$ where $\nu_\al$ is a $U$-invariant ergodic probability measure (cf \cite{EL}).
    By Ratner's Theorem \ref{ra}, for $\nu$-a.e.~$\al$, the measure $\nu_\al$ is an $L_\alpha$-invariant probability measure supported on a closed orbit $x_\al L_\al$ for some $x_\alpha\in \Ga\ba G$ and some connected closed subgroup $L_\al<G$ containing $U$.

    Let $E$ be the set of all $\al$ such that $L_\al\ne G$.
    Then for all $\al\in E$, we have $$x_\al L_\al\subset\mathscr{S}(U) .$$
 Since $$\nu(E)=\int_{\alpha\in \Ga\ba G} \nu_\al(E)\,d\nu(\alpha) \leq \nu(\mathscr{S}(U)) ,$$
   it follows from \eqref{si} that $\nu(E)=0$.
Therefore $$L_\al=G\quad\text{for $\nu$-a.e. $\al$}.$$ Hence $\nu$ is the $G$-invariant measure. Since the support of $\nu$ is contained in $\overline{xU}$, it follows that $\overline{xU}=\Ga\ba G$.
As $\nu(\Ga\ba G)=1$, this forces 
$\Ga$ to be a lattice in $G$.
\end{proof}

To complete the argument, we recall the following version of Ratner's orbit-closure theorem for finite-volume homogeneous spaces, which will be used to analyze the structure of orbit closures arising in the previous proposition.

\begin{theorem}[Ratner {\cite[Theorem 4]{RaGF}}]\label{R}
    Let $L$ be a connected linear Lie group and
    $\Delta<L$ be a lattice. For every connected
    closed subgroup $W<L$ generated by one-parameter unipotent subgroups in it and $x\in \Delta\ba L$,
    there exists a connected closed subgroup $L_0$ containing $W$ such that
    $\overline{xW}=xL_0$ and $\op{Stab}_{L_0}(x)$ is a lattice in $L_0$. Moreover there is a one-parameter unipotent subgroup $V<W$
    which acts ergodically on $(xL_0, \mu_{L_0})$ where $\mu_{L_0}$ is the $L_0$-invariant probability measure on $xL_0$.
\end{theorem}

With all the necessary dynamical ingredients in place, namely, the avoidance theorem, the non-divergence result within $\RFM$
 and Ratner's orbit-closure theorem, we can now prove Theorem \ref{Theoremmax}.

\begin{proof}[Proof of Theorem  \ref{Theoremmax}] Suppose that $xW\subset \RFM$.
First, consider the case when $W$ is a one-parameter unipotent subgroup. 
By Proposition \ref{prop.iso}, either $\overline{xW}=\Ga\ba G$ and
$\Vol(\Ga\ba G)<\infty$, or $x\in \mathscr S(W)$. Assume now that $x\in \mathscr S(W)$.
Then there exists a minimal connected closed subgroup $L<G$ containing $W$ such that $xL$ is closed and carries an $L$-invariant probability measure.  We claim that $\overline{xW}=xL$.
Let $g\in G$ with $x=\Ga g$.
Then $\op{Stab}_L(x)= L\cap g^{-1}\Ga g $ is a lattice in $L$.
The map $x\ell \mapsto [e]\ell$ defines an $L$-equivariant homeomorphism
 $xL\simeq (L\cap g^{-1}\Ga g)\ba L $ for $[e]$ denotes the identity coset in $L\cap g^{-1}\Ga g\ba L$. 
 By Theorem \ref{R},
 $$\overline{[e]W} = [e]L_0$$
 for some connected closed subgroup $L_0$ containing $W$
 and $W$ acts ergodically on $([e]L_0,\mu_{L_0})$. By the minimality of $L$, it follows that
 $L=L_0$, proving that $\overline{xW}=xL$.
 
 The general case can be deduced from this following Ratner \cite{RaGF}. For the reader's convenience, we recall Ratner's argument, which is slightly simpler in the present rank-one setting. Suppose that $W$ is a connected unipotent subgroup. Let $\cal U$ denote the set of all one-parameter subgroups of $W$. 
 By the previous case, for each $V\in \cal U$, there exists a closed connected subgroup $L(V)<G$ such that $V\subset L(V)$ and
 $\overline{xV}=xL(V)$ and $V$ acts ergodically on $(xL(V), \mu_{L(V)})$.
 We then have $$W=\bigcup \{W\cap L(V): V\in \cal U\},$$
 which is a countable union
 by Theorem \ref{ragf}. By the Baire category theorem, it follows that $W\subset L(V)$
 for some $V\in \cal U$. Hence $\overline{xW}=xL(V)$, proving the claim for unipotent $W$.
 Now consider a general case.  Since $G$ has rank one,
 $W$ is either unipotent or a simple non-compact closed subgroup.
 As the former case has already been treated, we may assume that $W$ is a connected simple non-compact subgroup. Then $W$ is generated by a pair of opposite maximal unipotent subgroups in it, say, $V^+$ and $V^-$. By applying the previous case to $ x wV^+w^{-1}\subset \RFM$
 for each $w\in V^-$, and using the countability result (Theorem \ref{ragf}),
 we obtain a closed connected subgroup $L<G$ and
 a subset $S\subset V^-$ of positive Haar measure such that, for all $w\in S$, $wV^+w^{-1}\subset L$ and
 $$\overline{xwV^+w^{-1}}= xL $$
 and $\op{Stab}_x L$ is a lattice in $L$. 
The set of elements $w\in V^-$ for which the closed group generated by $V^+$ and $wV^+w^{-1}$ is a proper algebraic subgroup of $W$ is contained in a countable union of proper algebraic subsets of $V^-$.
Since $S$ has positive Haar measure in $V^-$, 
it follows that $W\subset L$, and hence
 $\overline{xW}=xL$. Since $L$ admits a lattice, it is unimodular. Hence
 the ``moreover'' part follows from Lemma \ref{pppp}.
 This completes the proof. \end{proof}
 
\begin{rmk} \rm
In the proof of Theorem \ref{Theoremmax}, geometric finiteness is used to prevent escape of mass, while rank one is used in the passage from the one-parameter case to a general subgroup generated by unipotent one-parameter subgroups. If $G$ is an arbitrary connected linear Lie group and $xW$ is bounded, nonescape of mass is automatic. The one-parameter argument above then applies, and the passage to a general $W$ follows from Ratner's general argument \cite[proof of Theorem~4]{RaGF}, in place of the rank-one reduction used above. Together with Theorems \ref{ra}, \ref{lem.DM2}, and \ref{R}, this yields:
 \begin{theorem}\label{Theoremmaxfin}
Let $G$ be a connected linear Lie group and $\Ga<G$ a discrete subgroup.
Let $W<G$ be a connected closed subgroup generated by one-parameter unipotent subgroups. If $xW$ is bounded for $x\in \Ga\ba G$,
then there exists a connected  Lie subgroup $L$ containing $W$ such that
$$\overline{xW} =xL$$
and $ \op{Stab}_L(x)$ is a lattice in $L$.
\end{theorem}
\end{rmk}

\section{Equidistribution in $\RFM$}\label{sec.equidistribution}
In this section, we establish an equidistribution theorem for the homogeneous measures associated with a sequence of pairwise distinct maximal proper totally geodesic submanifolds $\cN_i=x_iH_i.o$ contained in $\core(\cM)$, where $x_i\in\Ga\ba G$ and $H_i\in\mathscr H$. Combined with the orbit-closure classification from Section~\ref{sec.orbitclosures}, this result provides the key dynamical ingredient in the proof of the  finiteness theorem (Theorem \ref{fin1}) in Section~\ref{sec.rigidity}.

We will use the following equidistribution theorem of Mozes-Shah: 
\begin{thm}[Mozes-Shah {\cite[Theorem 1.1]{MS}}] \label{prop.MS1} Let $G$ be a connected Lie group and  $\Ga<G$ be a discrete subgroup. Let $U_i$ be a sequence of one-parameter unipotent subgroups of $G$.
 Assume that there exists an infinite sequence of $U_i$-invariant ergodic probability measures $\nu_i$ on $\Ga\ba G$ converging to a probability measure $\nu$ on $\Ga\ba G$ as $i\to\infty$ in the weak-$^*$ topology. Let $x\in \supp \nu$.
Then the following holds:
\begin{enumerate}
    \item $\supp \nu =xL$ where $L=\{g\in G: \nu. g=\nu\}$\footnote{$\nu .g (E)=\nu (Eg)$ for a Borel subset $E\subset \Ga\ba G$};
\item  Let  $g_i'\to e$ be a sequence
in $G$ as $i\to\infty$  such that for all $i\in \mathbb N$,
$xg_i'\in \supp \nu_i$ and $\{xg_i' U_i\}$ is uniformly distributed \footnote{We say that $xU$ is uniformly distributed with respect to $\nu$ if the sequence of measures $\frac{1}{T}\int_0^T \delta_{xu_s} ds$ converges to $\nu$ as $T\to \infty$} with respect to $\nu_i$.
Then there exists $i_0\ge 1$ such that  
$$\supp \nu_i\subset (\supp\nu).g_i'\quad\text{for all $i\geq i_0$.}$$
    \item $\nu$ is invariant and ergodic for the action of the subgroup generated by $\{g_i'U_ig_i'^{-1}:i\geq i_0\}$.
\end{enumerate}
\end{thm}

The Mozes--Shah theorem provides the fundamental tool for analyzing weak $^*$ limits of invariant probability measures arising from unipotent orbits. To apply it effectively to our setting, we first record a structural lemma describing 
the closure of an $H^{\mathrm{nc}}$-orbit inside a closed $H$-orbit.
\begin{lem}\label{lem.PE2} 
Let $\Ga<G$ be a discrete subgroup.
    Suppose that $xH$ is closed for some $x\in\Ga\ba G$ and $H\in\mathscr H$.
    Then there exists a connected closed subgroup $F<G$ such that $H^{\op{nc}}\subset F\subset H$ and $$\overline{xH^{\op{nc}}}=xF. $$
\end{lem}
\begin{proof}
Choose $g\in G$ such that $x=[g]$. After replacing $\Ga$ by $g^{-1}\Ga g$, we may assume that $x=[e]$. Put
$H_0:=H^{\op{nc}}
$ and $C:=\op{C}_G(H_0)^\circ$.
Since $H=\op{N}_G(H_0)^\circ$, the multiplication map
$$
j:S:=H_0\times C\to H
$$
is a surjective homomorphism with finite kernel. Let $\pi_1$ and $\pi_2$ be the two coordinate projections and set
$\Ga_0:=j^{-1}(\Ga\cap H)$.
We first justify the computation 
\begin{equation}\label{eq.nc}
\overline{\Ga_0(H_0\times\{e\})}
   =H_0\times\overline{\pi_2(\Ga_0)}.
\end{equation}
Indeed, if $\gamma=(h_\gamma,c_\gamma)\in\Ga_0$, then
$\gamma(H_0\times\{e\})=H_0\times\{c_\gamma\}$.
Taking the union over $\gamma\in\Ga_0$ gives
$\Ga_0(H_0\times\{e\})=H_0\times\pi_2(\Ga_0)$,
and \eqref{eq.nc} follows by taking closures.
Set
$$
D:=\overline{\pi_2(\Ga_0)}
\quad\text{and}\quad
F:=j(H_0\times D^\circ).
$$
Since $j$ has finite kernel, it is proper. Hence $F$ is a connected closed subgroup satisfying $H_0\subset F\subset H$.
We next verify the orbit-closure assertion. Since $\ker j\subset\Ga_0$, the map $j$ induces a homeomorphism
$
\Ga_0\ba S\to (\Ga\cap H)\ba H.
$
Moreover, since $[e]H$ is closed, Lemma \ref{properr} shows that the canonical map
$$
(\Ga\cap H)\ba H\to\Ga\ba G
$$
is proper. Consequently, the induced map
$\iota:\Ga_0\ba S\to\Ga\ba G
$ is proper.
Equation \eqref{eq.nc} implies
$$
\overline{[e](H_0\times\{e\})}
   =[e](H_0\times D)
\quad\text{in }\Ga_0\ba S.
$$
We claim that
$$
[e](H_0\times D)=[e](H_0\times D^\circ).
$$
Indeed, every connected component of $D$ meets the dense subgroup $\pi_2(\Ga_0)$. Given $d\in D$, choose $\gamma=(h_\gamma,d_\gamma)\in\Ga_0$ such that $d_\gamma$ and $d$ belong to the same connected component of $D$. Then $d_\gamma^{-1}d\in D^\circ$, and
$$
[e](h,d)
 =[e]\gamma^{-1}(h,d)
 =[e](h_\gamma^{-1}h,d_\gamma^{-1}d)
 \in[e](H_0\times D^\circ).
$$
The reverse inclusion is immediate. Applying the proper map $\iota$ now gives
$$
\overline{[e]H_0}
 =[e]j(H_0\times D^\circ)
 =[e]F,
$$
as desired.
\end{proof}

We are now ready to combine the preceding results with the Mozes--Shah theorem to obtain the desired equidistribution statement:
\begin{thm}\label{eq00}
Let $\Ga<G$ be geometrically finite. Suppose that
$$
\cN_i=x_iH_i.o\subset\core(\cM),
\qquad x_i\in\Ga\ba G,\quad H_i\in\mathscr H,
$$
are pairwise distinct maximal proper totally geodesic submanifolds. Then there exist connected closed subgroups $F_i<G$ such that
$$
H_i^{\nc}\subset F_i,\qquad
\overline{x_iH_i^{\nc}}=x_iF_i,
$$
and $\op{Stab}_{F_i}(x_i)$ is a lattice in $F_i$. For every $i$, either $F_i=G$ or $F_i\subset H_i$. Moreover, $\Ga$ is a lattice and the $F_i$-invariant probability measures $\mu_{F_i}$ converge to the $G$-invariant probability measure $\mu_G$ on $\Ga\ba G$.
\end{thm}
\begin{proof}
By Lemma \ref{vol}, $x_iH_i\subset\RFM$. Since $H_i^{\nc}$ is generated by one-parameter unipotent subgroups, Theorem \ref{Theoremmax} gives a connected closed subgroup $F_i<G$ such that
$$
H_i^{\nc}\subset F_i,\qquad
\overline{x_iH_i^{\nc}}=x_iF_i,
$$
and $\op{Stab}_{F_i}(x_i)$ is a lattice in $F_i$. Let $\mu_i:=\mu_{F_i}$ be the corresponding invariant probability measure. Since $F_i$ contains the non-compact simple group $H_i^{\nc}$, Lemma \ref{pppp} implies that $F_i$ is reductive. By Lemma \ref{lem.e1},
$F_i$ is $\Theta$-invariant.
Suppose first that $F_i\ne G$, and set
$J_i:=\op{N}_G(F_i^{\nc})^\circ.$
Then $J_i\in\mathscr H$ and $J_i\ne G$. Since $x_iF_i\subset\RFM$ and $J_i.o=F_i.o$, Lemma \ref{vol} gives $x_iJ_i\subset\RFM$. Moreover,
$$
H_i.o=H_i^{\nc}.o\subset F_i^{\nc}.o=J_i.o.
$$
Hence $\cN_i=x_iH_i.o\subset x_iJ_i.o$, and by maximality, $\cN_i=x_iJ_i.o$.

We claim that $J_i=H_i$. Since $\cN_i=x_iJ_i.o$,
we have $$
g_iJ_i.o\subset\bigcup_{\ga\in\Ga}\ga g_iH_i.o
$$
for every representative $x_i=[g_i]$. By the Baire category theorem, one of the intersections $g_iJ_i.o\cap\ga g_iH_i.o$ has non-empty interior in $g_iJ_i.o$. Therefore
$$
\dim J_i.o\leq\dim H_i.o.
$$
Together with the inclusion $H_i.o\subset J_i.o$, this gives $H_i.o=J_i.o$. Lemma \ref{gy} now yields $H_i=J_i$. Since the connected group $F_i$ normalizes $F_i^{\nc}$, we conclude that $F_i\subset J_i=H_i$.
In particular, whenever $F_i\ne G$,
\begin{equation}\label{eq.FH}
F_i^{\nc}=H_i^{\nc}
\qquad\text{and}\qquad
x_iF_i.o=x_iH_i.o=\cN_i.
\end{equation}

We now prove that $\mu_i\to\mu_G$ in weak-* topology as $i\to\infty$. 
By passing to a subsequence, we may assume that
 $\mu_i$ weakly converges to a probability measure
 $\mu$ on the one-point compactification of $\Gamma\backslash G$. It suffices to show that $\mu=\mu_G$.
If $F_i=G$ for infinitely many $i$, then $\op{Stab}_G(x_i)$ is a lattice, so $\Ga$ is a lattice and $\mu_i$ is the $G$-invariant probability measure and hence  $\mu=\mu_G$. We may therefore assume that $F_i\ne G$ for every $i$. Thus \eqref{eq.FH} holds.

Choose a one-parameter unipotent subgroup
$$
U_i=\{u_{i,s}:s\in\br\}<H_i^{\nc}.
$$
By \cite[Proposition~2.1]{MS}, $H_i^{\nc}$ acts ergodically on $(x_iF_i,\mu_i)$. Since $H_i^{\nc}$ is simple, the Mautner phenomenon implies that $U_i$ acts ergodically on $(x_iF_i,\mu_i)$. Hence there is a $\mu_i$-conull set $E_i\subset x_iF_i$ such that every point of $E_i$ is uniformly distributed with respect to $\mu_i$ under $U_i$.

Choose $\epsilon>0$ for which the decomposition \eqref{eq.tt} holds with $2\epsilon$ in place of $\epsilon$. By Lemma \ref{complete}, $x_iF_i.o$ meets $\cM_{\geq 2\epsilon}$. Consequently,
$$
\Omega_i:=\{y\in x_iF_i:y.o\in\cM_{>\epsilon}\}
$$
is a non-empty open subset of $x_iF_i$, and therefore $\mu_i(\Omega_i)>0$. Thus $E_i\cap\Omega_i\ne\emptyset$. Replacing $x_i$ by an element of $E_i\cap\Omega_i$,
we may assume that $x_i$ belongs to the fixed compact set
$$
\cal K:=\pi^{-1}\bigl(\core(\cM)\cap\cM_{\geq\epsilon}\bigr)
 \subset\Ga\ba G,
$$
where $\pi:\Ga\ba G\to\cM$ is the basepoint projection.
Choose a subsequence $T_i\to\infty$ such that
$$
\nu_i:=\frac1{T_i}\int_0^{T_i}\delta_{x_iu_{i,s}}\,ds\to \mu .
$$
 Proposition \ref{ND} gives
$\mu(\RFM)=1.$ In particular, $\mu$ is a probability measure on $\Ga\ba G$. We may now apply Theorem \ref{prop.MS1}. Write
$$
\supp\mu=xL,
\qquad
L=\{g\in G:\mu.g=\mu\},
$$
and choose representatives $x=[p_0]$ and $x_i=[g_i]$. Since $x\in\supp\mu$ and $x_iU_i$ is uniformly distributed with respect to $\mu_i$, there exist $\ga_i\in\Ga$ and $s_i\in\br$ such that
$$
q_i:=\ga_i g_i u_{i,s_i}\to p_0.
$$
Put $g_i':=p_0^{-1}q_i$. Then $g_i'\to e$ and $xg_i'U_i=x_iU_i$. Theorem \ref{prop.MS1}(2) therefore gives, for all sufficiently large $i$,
\begin{equation}\label{LL}
x_iF_i=\supp\mu_i\subset xLg_i'.
\end{equation}
Thus
$$
\Ga q_iF_iq_i^{-1}\subset\Ga p_0Lp_0^{-1}.
$$
 Since $xL$ is closed, the canonical map
$$
(\Ga\cap p_0Lp_0^{-1})\ba p_0Lp_0^{-1}
   \to\Ga\ba G
$$
is a homeomorphism onto its image. Since $q_iF_iq_i^{-1}$ is a connected subgroup, we obtain
\begin{equation}\label{eq.groupcontain}
q_iF_iq_i^{-1}\subset p_0L^\circ p_0^{-1}.
\end{equation}
We claim that $L^\circ$ is reductive. 

Since $\op{Stab}_L(x)$ is a lattice in $L$, the subgroup $L^\circ$ is unimodular. By \eqref{eq.groupcontain}, $L^\circ$ contains a conjugate of $F_i$. It cannot therefore be a compact extension of a unipotent subgroup.  By Lemma \ref{pppp}, the claim follows.
Set
$$
L_i:=q_i^{-1}p_0L^\circ p_0^{-1}q_i.
$$
By \eqref{eq.groupcontain}, $F_i\subset L_i$. Lemma \ref{lem.e1} implies that $L_i$ is $\Theta$-invariant. Moreover, $[q_i]L_i.o$ has finite volume. Hence by Lemma \ref{vol},
$$
[q_i]L_i.o\subset\core(\cM).
$$
Since $u_{i,s_i}\in F_i$ and \eqref{eq.FH} holds,
$$
\cN_i=x_iF_i.o=[q_i]F_i.o\subset[q_i]L_i.o.
$$
By maximality,
$$
\text{either }L_i=G
\quad\text{or}\quad
[q_i]F_i.o=[q_i]L_i.o.
$$
In the second case, the same Baire-category argument used above, together with $F_i.o\subset L_i.o$, gives
$F_i.o=L_i.o=H_i.o.$
Since $L_i$ is connected, Lemma \ref{gy} then gives
$F_i\subset L_i\subset H_i$.

We claim that $L_i=G$ for all sufficiently large $i$. Otherwise, 
$$
q_iH_i^{\nc}q_i^{-1}
 \subset p_0L^\circ p_0^{-1}
 \subset q_iH_iq_i^{-1}
$$
for infinitely many $i$.
Since $p_0L^\circ p_0^{-1}$ is connected and reductive, its non-compact semisimple part is
$$
\bigl(p_0L^\circ p_0^{-1}\bigr)^{\nc}
   =q_iH_i^{\nc}q_i^{-1}.
$$
Therefore
$$
q_iH_iq_i^{-1}
 =\op{N}_G\!\left(
   \bigl(p_0L^\circ p_0^{-1}\bigr)^{\nc}\right)^\circ,
$$
which is independent of $i$. Since $u_{i,s_i}\in H_i$, this says that all the groups $
\ga_i g_iH_ig_i^{-1}\ga_i^{-1}$
are equal. By Lemma \ref{gy}, the totally geodesic subspaces $\ga_i g_iH_i.o$ are then equal, and hence $\cN_i$ are equal. This contradicts the assumption that $\cN_i$ are pairwise distinct.
It follows that $L^\circ=G$, and hence $L=G$. Thus $\Ga$ is a lattice and $\mu=\mu_G$.
\end{proof}

\section{Rigidity of totally geodesic submanifolds in the convex core}\label{sec.rigidity}
In this section, we deduce the rigidity, properness, and finiteness results for totally geodesic submanifolds contained in the convex core of $\cM$ from the dynamical statements proved in Sections~\ref{sec.orbitclosures} and~\ref{sec.equidistribution}. 

\subsection*{Rigidity} We begin with the topological rigidity of totally geodesic submanifolds:
\begin{theorem}\label{Theoremmax0} Let $\cM$ be geometrically finite. If $\cN$ is a totally geodesic immersed submanifold of dimension at least two contained in  $\op{core}(\cal M)$, 
then the closure of $\cN$ is a totally geodesic, properly immersed submanifold of finite volume.\end{theorem}
\begin{proof} Let $\cN$ be a totally geodesic immersed submanifold of dimension at least two.
By Lemma \ref{lem.PE0}, $$\cN=xH.o\quad\text{ for some $H\in \mathscr H^*$ and $x\in \Ga\ba G$.}$$ Suppose that $\cal N\subset \core (\cM)$. By
Lemma \ref{vol}, $xH\subset \RFM.$
By the definition of $\mathscr H^*$, we have $H=H^{\op{nc}}$ and hence $H$ is generated by one-parameter unipotent subgroups. Therefore, by Theorem \ref{Theoremmax}, $\overline{xH}= xL$
for some connected closed subgroup $L$ containing $H$ such that $\op{Stab}_L(x)$ is a lattice in $L$.
By Lemma \ref{lem.e1}, $L$ is $\Theta$-invariant and hence $L^{\op{nc}}\in \mathscr{H}^*$. Consequently, $$xL.o=xL^{\op{nc}}.o$$ is a properly immersed totally geodesic submanifold. Since $\overline{xH}=xL$, we conclude that $\overline{\cN}= xL.o$, and  that $xL.o$ has finite volume. 
\end{proof}

We remark that this type of rigidity fails for a general totally geodesic submanifold  not contained in the core; see \cite{MMO1} for an example of quasifuchsian manifolds which contain geodesic planes whose closures are not submanifolds.

\subsection*{Properness} When the ambient manifold $\cM$ has infinite volume, Theorem \ref{Theoremmax0} immediately yields Theorem \ref{p} from the introduction:
\begin{proof}[Proof of Theorem \ref{p}]
   By Theorem \ref{Theoremmax0}, $\overline{\cal N}$ is a totally geodesic properly immersed submanifold of finite volume. Since $\overline{\cal N}$ is contained in $\core (\cal M)$, the maximality of $\cal N$ implies $\cal N=\overline{\cal N}$, completing the proof.
\end{proof}

For geometrically finite real hyperbolic manifolds, every properly immersed totally geodesic submanifold of $\cM$
is itself geometrically finite \cite{OS1}. In particular, every properly immersed $\cN$ contained in $\core (\cM)$ must have finite volume.
A direct consequence of Theorem \ref{Theoremmax0} shows that this phenomenon extends to all geometrically finite rank-one manifolds:
\begin{cor}\label{gfvol}
    Let $\cM$ be geometrically finite. Then every properly immersed totally geodesic submanifold contained in $\core (\cM)$ has finite volume.
\end{cor}

\subsection*{Equidistribution} Building on the rigidity results above, we now examine whether infinitely many maximal totally geodesic submanifolds can exist inside the convex core. The following theorem shows that this 
phenomenon occurs only in the finite-volume case.

For a finite-volume geodesic submanifold $\cN=xH.o$ for $H\in \mathscr H$, the normalized volume measure $\mu_{\cal N}$ is the probability measure supported on $\cal N\subset \cM$, defined as the push-forward  $\pi_*(\mu_H)$  of the $H$-invariant probability measure $\mu_H$ under the projection $\pi:(\Ga\cap H)\ba H\to \cal M=\Ga \ba G.o$. Similarly, when $\cM$ has finite volume, we denote by $\mu_{\cal M}$
the probability measure on $\cM$ which is the push-forward of the $G$-invariant probability measure of $\Ga\ba G$.

\begin{thm}\label{eq}
    Let $\cM$ be a geometrically finite manifold. If there exist infinitely many maximal totally geodesic submanifolds $\cN_i$ contained in $\core (\cM)$ of dimension at least two,
    then $$\Vol (\cM)<\infty.$$
If, moreover,  each $\mathcal N_i$ has finite volume, then  the normalized volume measures $\mu_{\cal N_i}$ 
 equidistribute in $\cM$ as $i\to \infty$, that is, for every $f\in C_c(\cM)$,
    $$\lim_{i\to \infty} \int_{\cN_i} f(x) d\mu_{\cN_i} (x) = \int_{\cM} f (x) d\mu_{\cM}(x) .$$
\end{thm}
\begin{proof}
By Lemma \ref{lem.PE0}, write $\cN_i=x_iH_i.o$
for some $H_i\in\mathscr H$ and $x_i\in\Ga\ba G$. Then, by Theorem
\ref{eq00}, $\Vol(\cM)<\infty$.
Suppose now, in addition, that $\Vol(\cN_i)<\infty$ for every $i$.
By Lemma \ref{vol}, $x_iH_i$ carries a finite
$H_i$-invariant measure, and hence $x_iH_i$ is closed by
\cite[Theorem~1.13]{Rag}. Let $F_i$ be the subgroup given by
Theorem \ref{eq00}. Since $x_iH_i$ is closed,
\[
x_iF_i=\overline{x_iH_i^{\nc}}\subset x_iH_i.
\]
By Theorem
\ref{eq00}, $F_i\subset H_i$. Since
$H_i^{\nc}\subset F_i\subset H_i$ and $H_i^{\nc}.o=H_i.o$,
\[
H_i^{\nc}.o=F_i.o=H_i.o.
\]

Let $\pi:\Ga\ba G\to\cM$ be the basepoint projection. It follows that
\[
\pi_*\mu_{F_i}=\pi_*\mu_{H_i}=\mu_{\cN_i},
\qquad
\pi_*\mu_G=\mu_{\cM}.
\]
Since $K$ is compact, $\pi$ is proper. Hence pushing forward the
weak-$^*$ convergence $\mu_{F_i}\to\mu_G$ gives
$\mu_{\cN_i}\to\mu_{\cM}$, as desired.
\end{proof}

\subsection*{Finiteness}  Combining Theorem \ref{Theoremmax0} and Theorem \ref{eq}, we  obtain the following finiteness result for infinite-volume manifolds.

\begin{cor}\label{Theoremfin3}
Let $\cal M$ be geometrically finite with $\op{Vol}(\cM)=\infty$.  Then there exist only finitely many maximal  totally geodesic submanifolds in $\cm$ of dimension at least two. \end{cor}
Indeed, Theorem \ref{p} first shows that every such maximal submanifold has finite volume, so Theorem \ref{eq} applies if there are infinitely many.

Having established the finiteness of totally geodesic submanifolds in the convex core, we now prove Theorem \ref{SX}, which reformulates this statement on the level of their ideal boundaries.
\subsection*{Proof of Theorem \ref{SX}}
Let $S=\partial_\infty Y\in \mathcal S_X$ be a maximal element contained in $\La$. Since $Y=\hull (\partial_\infty Y)$ by Lemma \ref{yb},
the subspace $Y$ is a maximal totally geodesic subspace inside $\hull (\La)$. 
We claim that $\Ga S$ is closed in $\mathcal S_X$ with respect to the Chabauty-Hausdorff  topology. To see this, suppose that $\ga_i S=\ga_i \partial_\infty Y$ converges to $S'=\partial_\infty Y' \in \mathcal S_X$. 
 We claim that as $i\to \infty$,
 $$\ga_i Y=\hull (\ga_i S)\to Y'=\hull (S').$$ To verify this,
 let $x_i\in \ga_i Y$ converge to some $x \in X$, and let $\ell_i$ be a geodesic passing through $x_i$. Since $x_i\to x$ and hence all $\ell_i$ pass through a fixed compact subset of $X$, and hence, after passing to a subsequence, $\ell_i$ converges to some geodesic $\ell$ in $X$. Since $x\in \ell$ and the endpoints of $\ell$ lie in $S'$,  we have $x\in \hull (S')=Y'$. Therefore every Hausdorff limit of $\ga_i Y$ is contained in $Y'$. Conversely, let $x\in Y'$. Then $x$ lies on a geodesic $\ell$ with endpoints $\xi, \xi'$ in $S'$. Since $\ga_i S\to S'$, we can choose sequences $\xi_i\ne \xi'_i$ in $\ga_i S$
with $\xi_i\to \xi$ and $\xi'_i\to \xi'$. Let  $\ell_i$ denote the geodesic connecting $\xi_i$ and $\xi'_i$. Then $\ell_i$ converges to $\ell$, and
  we may choose a sequence $x_i\in \ell_i$ converging to $x$.
 Hence  $\ga_i Y\to Y'$ in the Chabauty-Hausdorff topology.

 Since $\Ga\ba \Ga Y$ is closed in $\cal M$, it follows that $\hull (S')\subset \Ga Y $, and hence $S'\in \Ga \partial_\infty Y$.
 This proves that $\Ga S$ is closed in $\mathcal S_X$.
 The finiteness then follows from Corollary \ref{Theoremfin3}. \qed

 \subsection*{Rigidity of bounded geodesic planes}
In the proof of Corollary \ref{Theoremfin3}, the
  geometric finiteness was used only in Proposition \ref{ND} to ensure that weak-$^*$ limits of the measures $\mu_i$ were probability measures in $\Ga\ba G$. When the corresponding totally geodesic submanifolds are contained in a fixed compact subset of $\Ga\ba G$, the same reasoning applies without assuming geometric finiteness:
\begin{theorem}\label{thm.boundedrigidity} Let $\cM=\Ga\ba X$ be a rank-one locally symmetric space which is non-compact. Then the closure of every bounded totally geodesic submanifold of dimension at least two is a totally geodesic submanifold. Moreover, every  compact subset of $\cM$ contains only finitely many maximal totally geodesic submanifolds of dimension at least two.
    \end{theorem}

 We record the following characterization of $\Ga$, up to finite index, in terms of totally geodesic subspaces contained in $\hull(\La)$.
\begin{theorem}\label{rig2}
    Let $\Ga<G$ be a geometrically finite Zariski dense non-lattice subgroup.
     Let $\mathsf T$ be any non-empty $\Ga$-invariant collection of
 maximal totally geodesic subspaces of dimension at least two contained in $\hull (\La)$.
 Then
$$\Ga \text{ has finite index in the subgroup }\Delta:=\{g\in G: g ({\mathsf T} )= {\mathsf T}\}.$$  
\end{theorem}
\begin{proof}
Since
$$
\mathcal B:=\overline{\bigcup_{Y\in\mathsf T}\partial_\infty Y}
\subset\La.
$$
is non-empty $\Ga$-invariant closed subset of $\La$, $\mathcal B=\La$. 
Note that
$$
\Ga<\Delta<D:=\{g\in G:g\La=\La\}.
$$

We first show that $D$ is discrete. The group $D$ is closed. Let $\frak d$ be the Lie algebra of $D^\circ$. Since $D^\circ$ is normal in $D$ and $\Ga<D$, the subspace $\frak d\subset\frak g$ is $\Ad(\Ga)$-invariant. The stabilizer of $\frak d$ under the adjoint action is a real algebraic subgroup of $G$. The Zariski density of $\Ga$ therefore implies that $\frak d$ is $\Ad(G)$-invariant, and hence is an ideal of $\frak g$. By the simplicity of $G$, either $\frak d=\{0\}$ or $\frak d=\frak g$. The latter would imply $D^\circ=G$ and hence $D=G$, which is impossible because $\La\ne\partial_\infty X$. Thus $\frak d=\{0\}$, proving that $D$ is discrete.

Let
$${\mathcal C}:=\{x\in X:d(x,\hull(\La))\leq1\}.
$$
The group $D$ preserves ${\mathcal C}$. By geometric finiteness,
$\Vol(\Ga\ba\ \mathcal C)<\infty$.
On the other hand, since $D$ is discrete and ${\mathcal C}$ has nonempty interior, $\Vol(D\ba{\mathcal C})>0$. 
Since
$$
\Vol(\Ga\ba {\mathcal C})
=[D:\Ga]\cdot\Vol(D\ba {\mathcal C}).
$$
It follows that $[D:\Ga]<\infty$, and hence $[\Delta:\Ga]<\infty$.
\end{proof}

\end{document}